\documentclass[12pt]{amsart}

\usepackage{amsmath}
\usepackage[all]{xy}
\usepackage[margin=1in]{geometry}
\usepackage{mathrsfs}
\usepackage{setspace}

\usepackage{tikz}
\usetikzlibrary{calc}
\usetikzlibrary{arrows}

\usepackage{microtype}

\usepackage{amssymb}
\usepackage{amsthm}
\usepackage{thmtools}
\declaretheorem[numberwithin=section]{theorem}
\declaretheorem[sibling=theorem]{corollary}
\declaretheorem[sibling=theorem]{lemma}
\declaretheorem[sibling=theorem]{proposition}
\declaretheorem{question}
\declaretheorem{conjecture}
\declaretheorem[style=definition]{definition}
\declaretheorem[style=remark,numbered=no]{remark}

\declaretheorem[style=definition,numberwithin=section]{example}

\newcommand{\abs}[1]{\left\lvert #1 \right\rvert}
\newcommand{\set}[1]{\left\{ #1 \right\}}
\renewcommand{\P}{\mathbb P}

\newcommand{\C}{\mathbb C}

\newcommand{\Z}{\mathbb Z}

\newcommand{\fm}{\mathfrak m}
\newcommand{\cO}{\mathscr O}

\DeclareMathOperator{\im}{im}
\DeclareMathOperator{\rank}{rank}
\DeclareMathOperator{\Ass}{Ass}

\DeclareMathOperator{\len}{len}
\DeclareMathOperator{\Aut}{Aut}
\DeclareMathOperator{\Bl}{Bl}
\DeclareMathOperator{\codim}{codim}
\DeclareMathOperator{\id}{id}
\DeclareMathOperator{\Hilb}{Hilb}
\DeclareMathOperator{\mult}{mult}
\DeclareMathOperator{\Nef}{Nef}

\DeclareMathOperator{\GL}{GL}

\DeclareMathOperator{\Spec}{Spec}

\DeclareMathOperator{\Tor}{Tor}

\theoremstyle{remark}

\theoremstyle{definition}

\newcommand\nc{\newcommand}
\nc\on{\operatorname}

\newcommand\fp{{\mathfrak p}}
\newcommand\fq{{\mathfrak q}}

\nc\Hom{\on{Hom}}
\nc\Sections{\on{Sections}}
\nc\Specm{\on{Specm}}
\nc\ul{\underline}
\nc{\dfp}{\overset{\cdot}{\fp}}
\nc{\dfq}{\overset{\cdot}{\fq}}
\nc{\dfm}{\overset{\cdot}{\fm}}

\title[Dynamical {Mordell--Lang} and Automorphisms of Blow-ups]{Dynamical Mordell--Lang and \\ Automorphisms of Blow-ups} 
\author{John Lesieutre and Daniel Litt}

\begin{document}

\begin{abstract}
We show that if \(\phi : X \to X\) is an automorphism of a smooth projective variety and \(D \subset X\) is an irreducible divisor for which the set \(\set{d \in D : \text{$\phi^n(d) \in D$ for some nonzero $n$}}\) is not Zariski dense in $D$, then \((X, \phi)\) admits an equivariant rational fibration to a curve.  As a consequence, we show that certain blowups (e.g. blowups in high codimension) do not alter the finiteness of $\text{Aut}(X)$, extending results of Bayraktar--Cantat.  We also generalize results of Arnol'd on the growth of multiplicities of the intersection of a variety with the iterates of some other variety under an automorphism.

These results follow from a non-reduced analogue of the dynamical Mordell-Lang conjecture. Namely, let \(\phi : X \to X\) be an \'etale endomorphism of a smooth projective variety \(X\) over a field $k$ of characteristic zero.  We show that if \(Y\) and \(Z\) are two closed subschemes of \(X\), then the set \(A_\phi(Y,Z) = \set{n : \phi^n(Y) \subseteq Z}\) is the union of a finite set and finitely many residue classes, whose modulus is bounded in terms of the geometry of \(Y\).
\end{abstract}

\maketitle

\section{Introduction}
\label{introsection}

The aim of this paper is to study the group of biregular automorphisms of a projective variety \(X\), a basic invariant of \(X\).  The study of these groups naturally breaks up into the case of varieties of general type, which are well understood, and varieties not of general type.

Varieties of general type have finite birational and hence biregular automorphism groups; work of Hacon, McKernan, and Xu \cite{haconmckernanxu} gives bounds on their size.  On the other hand, varieties of smaller Kodaira dimension (and in particular rational varieties) may have wild birational automorphism groups, and little is known in general about possible groups of regular automorphisms.  For example, it seems to be unknown whether the group of components of \(\Aut(X)\) for \(X\) a rational surface is necessarily finitely generated.

As the birational automorphism group of a variety which is not of general type may be extremely large, it is natural to study the extent to which $\text{Aut}(X)$ can change under birational modifications of $X$, e.g. blowups.  One result of this paper is a sort of ``purity theorem'' for the automorphism group---namely, we show that its finiteness is not affected by blow-ups along smooth centers that have either large codimension or Picard rank \(1\).  These considerations naturally motivate a study of the dynamics of (exceptional) divisors under the action of an automorphism.  We study this via a generalization of the dynamical Mordell-Lang conjecture.  In our view, the main contribution of this paper is the insight that this (non-reduced) dynamical Mordell-Lang conjecture places serious restrictions on the automorphism groups of blowups of a variety and the dynamics of those automorphisms.

The paper proceeds as follows. First, in Sections~\ref{introsection} and \ref{examplesection}, we discuss the main results and illustrate them in the two-dimensional setting.
In Section \ref{mordelllang}, we formulate and prove a version of the dynamical Mordell--Lang conjecture for non-reduced schemes.  Sections~\ref{sectseparation} and \ref{sectglobalsep} then show how this result can be applied to study the orbits of divisors under infinite order automorphisms.  Finally, Sections \ref{secblowupautos} and \ref{highercodimstuff} explain how these results can be applied to questions about automorphisms of varieties: we extend results of Bayraktar--Cantat on automorphisms of blowups~\cite{bayraktarcantat} and show how some results of Arnol'd~\cite{arnold} can be generalized and understood as an instance of the non-reduced dynamical Mordell--Lang conjecture.

We first survey the main results. Throughout we work over an arbitrary algebraically closed field $k$, of characteristic zero. A variety is an integral scheme of finite type over \(k\), not assumed to be proper.  Many of our results hold without any properness assumptions, but some results requiring global geometric tools require projectivity.

The usual form of the Dynamical Mordell--Lang conjecture is the following.

\begin{conjecture}[Dynamical Mordell--Lang]
Suppose that \(\phi : X \to X\) is endomorphism of a quasiprojective variety \(X\) over \(k\), with \(V \subset X\) a subvariety and \(p \in X\) a point.  Then the set
\[
A_\phi(p,V) = \set{n : \phi^{n}(p) \in V}
\]
is a union of a finite set and a finite number of arithmetic progressions. 
\end{conjecture}

When \(\phi\) is \'etale, the conjecture is a theorem of Bell--Ghioca--Tucker~\cite{bellghiocatucker}.  Our first result extends this result in the \'etale case to the setting in which \(p\) and \(V\) are replaced by arbitrary closed subschemes of \(X\).

\begin{restatable}{theorem}{thmmordelllang}
\label{nonreducedmordelllang}
Suppose that \(X\) is a smooth variety over \(k\), and that \(\phi : X \to X\) is an \'etale endomorphism of \(X\).  Let \(Y\) and \(Z\) be closed subschemes of \(X\).  Then the set \(A_\phi(Y,Z) = \set{n : \phi^n(Y) \subseteq Z}\)
is a union of a finite set and a finite number of arithmetic progressions. 
\end{restatable}

This was proved by Bell--Lagarias for affine \(X\)~\cite{belllagarias}.  For our geometric applications, it will be important to have some control over how the lengths of the arithmetic progressions appearing in \(A_\phi(Y,Z)\) depend on \(Y\) and \(Z\).  

\begin{restatable}{theorem}{thmmordelllangpds}
\label{nonreducedmordelllangpds}
Let \(X\), \(\phi\), and \(Z\) be as in Theorem~\ref{nonreducedmordelllang}.
Suppose that \(Y_1^\circ,\ldots,Y_r^\circ\) is a finite set of closed subschemes of \(X\), defined by ideal sheaves $\mathscr{I}_{Y_i^\circ}$, and that \(\set{Y_j}\) is an infinite collection of subschemes defined by ideal sheaves of the form \[\sum_{i=1}^r \mathscr{I}_{Y_i^\circ}^{n_i},\] with $n_i\geq 0$.  Let \(\Xi\) be the set of all associated points of the \(Y_j\).
\begin{enumerate}
\item Suppose \(x\) is a point of \(X\) with \(x \in \bar{\xi}\) for only finitely many \(\xi \in \Xi\).  Then there exists a constant \(N = N(Y_j,Z,x)\) (independent of \(j\)) so that the set of \(n\) for which \(\phi^n(Y_j) \subseteq Z\) in $\on{Spec}(\mathscr{O}_{X,x})$ can be expressed as the union of arithmetic progressions of length bounded by \(N\) and a finite set.
\item Suppose that the set of associated points \(\Xi\) is finite.  Then there exists a constant \(N = N(\set{Y_j},Z)\) such that the lengths of the arithmetic progressions appearing in \(A_\phi(Y_j,Z)\) are bounded by $N$ (and in particular are independent of $j$).
\end{enumerate}
\end{restatable}
\begin{remark}
The behavior of the periods \(A_\phi(Y_j,Z)\) can be quite subtle when the set \(\set{Y_j}\) is infinite, and some hypotheses on associated points are necessary for (2) to hold. In Example~\ref{supportnotenough} below, we exhibit an automorphism \(\phi : X \to X\), a subscheme \(Z \subset X\), and an infinite set of closed subschemes \(Y_j\) which all have the same support, but for which the periods of \(A_\phi(Y_j,Z)\) grow without bound due to the presence of embedded points.
\end{remark}

We are primarily motivated by geometric applications of Theorem~\ref{nonreducedmordelllangpds}, which does not have an analog in the setting where \(Y\) and \(Z\) are reduced.  The uniform bound on the period makes it possible to control the multiplicities of the intersection of a subvariety \(V\) with the iterates \(\phi^n(V)\) under an automorphism.

If \(x\) is the generic point of a codimension \(2\) variety of \(X\), then the hypotheses of point (1) hold automatically: an associated point of \(Y_j\) whose closure contains \(x\) must be either \(x\) itself, or the generic point of a divisor containing \(x\).  But only finitely many divisors appear among the support of the \(Y_j\).
However, in general the set \(\Xi\) need not be finite; see Example~\ref{manyassociatedpoints} below.  Nevertheless, this hypothesis is satisfied in a number of natural geometric settings, discussed in Theorem~\ref{finiteassociate}.

The first geometric application of these results is:
\begin{restatable}{theorem}{thmlocalsep}
\label{separationthm}
Let \(X\) be a smooth variety over \(k\) and \(\phi : X \to X\) an automorphism.  Suppose that \(D \subset X\) is an irreducible divisor, containing an irreducible codimension \(2\) closed subset \(V \subset X\) with \(\phi(V) = V\).  If $D$ is not $\phi$-periodic, then there exists a projective birational morphism \(\pi  : Y \to X\) such that:
\begin{enumerate}
\item \(Y\) is smooth;
\item \(\pi\) is an isomorphism away from \(V\);
\item some iterate of \(\phi\) lifts to an automorphism \(\psi : Y \to Y\);
\item \(\pi(\psi^m(\tilde{D}) \cap \psi^n(\tilde{D}))\) does not contain \(V\) for any \(m \neq n\). 
\end{enumerate}
\end{restatable}

(Here $\tilde D$ is the strict tranform of $D$ on \(Y\).)

If the intersections of \(D\) with the divisors \(\phi^n(D)\) are a proper Zariski closed subset of \(D\), repeated applications of this local separation result make it possible to blow up \(X\) several times and simultaneously make the strict transforms of all of the divisors \(\phi^n(D)\) disjoint.  Global geometric considerations then yield:

\begin{restatable}{theorem}{thmglobalsep}
\label{mainblowup}
Let \(X\) be a smooth projective variety  and \(\phi : X \to X\) an automorphism.  Suppose that \(D \subset X\) is an irreducible divisor for which the set
\[
V(D) = \set{ d \in D: \text{$\phi^n(d) \in D$ for some nonzero $n$} } = \bigcup_{n \neq 0} D \cap \phi^n(D)
\]
is a proper, Zariski-closed subset of $D$.  Then there exists a projective birational morphism \(\pi : Y \to X\) such that
\begin{enumerate}
\item \(Y\) is smooth;
\item some iterate of \(\phi\) lifts to an automorphism \(\psi : Y \to Y\);
\item the divisors \(\psi^n(\tilde{D})\) are pairwise disjoint;
\item there exists a curve \(C\), a morphism \(f : Y \to C\), and an automorphism \(\tau : C \to C\) such that \(f \circ \psi = \tau \circ f\). 
\end{enumerate}
Moreover, the normal bundle of \(D\) satisfies \(H^0(D,N_{D/X}) > 0\) and \(D\) moves in a positive-dimensional family on \(X\).
\end{restatable}

Theorem~\ref{mainblowup} can then be used to address concrete questions about automorphisms of varieties constructed using sequences of blow-ups.

\begin{restatable}{theorem}{thmautomorphisms}
\label{generalblowup}
Suppose that \(X\) is a smooth projective variety of dimension \(n\), and that \(Y\) is a variety constructed by a sequence of smooth blow-ups of \(X\).  If each blow-up \(\pi_i : X_{i+1} \to X_i\) in the sequence, with exceptional divisor \(E\) and center of dimension \(r = \dim \pi_i(E)\), satisfies either
\begin{enumerate}
\item \(2r + 3 \leq n\), or
\item \(r + 3 \leq n\) and \(\Nef(E)\) is a polyhedral cone,
\end{enumerate}
then there exists an integer $N \geq 1$ such that for any automorphism $\phi$ of \(Y\), $\phi^N$ descends to an automorphism of $X$.
\end{restatable}

This was proved for blow-ups of type (1) by Bayraktar and Cantat under the additional hypothesis that \(\rho(X) = 1\)~\cite{bayraktarcantat}. Here we present a different proof which does not require this hypothesis.  Although the proof in case (1) is independent of our main results,  the handling of case (2) relies on the results of the preceding sections. 
Observe that a blow-up at a center of Picard rank $1$ automatically satisfies (2).
The fact that number $N$ in Theorem \ref{generalblowup} does not depend on $\phi$ is useful in applications, particularly Corollary \ref{boundedautos}.

Observe that part (2) of Theorem \ref{generalblowup} implies that blowups of fourfolds along points and curves cannot obtain new infinite order automorphisms:

\begin{restatable}{corollary}{corblowupauto}
\label{bcbound}
Suppose that \(X\) is a smooth projective variety of dimension \(4\), and that \(Y\) is a variety constructed by a sequence of blow-ups of points and smooth curves in \(X\).  Then there exists an integer $N \geq 1$ such that for any automorphism $\phi$ of \(Y\), $\phi^N$ descends to an automorphism of $X$.
\end{restatable}

The conclusion of Theorem \ref{generalblowup} is false when \(r = n-2\) for every value of \(n\) (see Example~\ref{ellipticfibr} below), but the situation for \(\lfloor \frac{n-1}{2} \rfloor \leq r \leq n-3\) remains unclear.  It seems natural to ask:

\begin{question}
Does there exist a projective variety \(X\) with \(\Aut(X)\) a finite group, and a sequence of smooth blow-ups \(\pi : Y \to X\) along smooth centers of codimension at least \(3\) such that \(\Aut(Y)\) is infinite?
\end{question}

Corollary~\ref{boundedautos} below shows that the answer is negative for \(\dim(X) \leq 4\).  

\begin{restatable}{corollary}{corboundedautos}
\label{boundedautos}
Let $X$ and $Y$ be as in Theorem \ref{generalblowup}.  Then if $\on{Aut}(X)$
has finite component group, the same is true for $\on{Aut}(Y)$.
\end{restatable}

Truong has also obtained a variety of results in this direction, under some additional hypotheses on the nef cone of \(Y\) (e.g.~\cite[Lemma 6]{truongblowup}).     
Arguments of a similar flavor appeared in the first author's paper~\cite{jdlthreeautos}, where the three-dimensional case of Theorem~\ref{separationthm} was proved by different methods.
The results on automorphisms in~\cite{jdlthreeautos} are specific to positive entropy automorphisms in dimension \(3\); much stronger conclusions can be obtained in that setting.
\section{Some examples}
\label{examplesection}

Before giving the proof of Theorem~\ref{nonreducedmordelllangpds}, we give an indication of its geometric content in the case that \(Y\) is a non-reduced subscheme.  For simplicity, suppose that \(X\) is a smooth algebraic surface, and that \(Y\) and \(Z\) are two smooth curves on \(X\) which intersect at \(V\), a fixed point of an automorphism \(\phi : X \to X\).  

Let \(Y^{(1)} \) be the first-order germ of \(Y\) at \(V\), a closed subscheme given by an embedding \(\Spec \C[\epsilon]/(\epsilon^2) \to X\).  We have \(\phi^n(Y^{(1)}) \subseteq Z\) exactly when \(\phi^n(Y)\) is tangent to \(Z\) at the fixed point \(V\). 

Theorem~\ref{nonreducedmordelllang} applied to the scheme \(Y^{(1)}\) then asserts that 
\begin{align*}
A_\phi(Y^{(1)},Z) &= \set{n : \phi^n(Y^{(1)}) \subseteq Z} \\
 &= \set{n : \text{$\phi^n(Y)$ is tangent to $Z$ at $V$}} 
\end{align*}
is a union of a finite set and finitely many arithmetic progressions.  

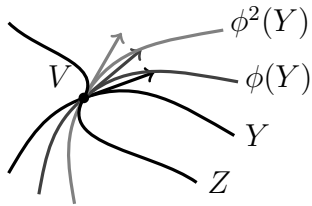
\begin{figure}[htb]
  \centering
\begin{tikzpicture}[scale=1.0]

\draw [very thick] (-1,-1) to[out=60,in=20-180] (0,0) to[out=20,in=180-30] (2,-1/2);
\node [right] at (2,-1/2) {$Y$};
\draw [->,very thick] (0,0) to ({cos(20)},{sin(20)});

\draw [rotate around={20:(0,0)}, color=black!75,very thick] (-1,-1) to[out=60,in=20-180] (0,0) to[out=20,in=180-30] (2,-1/2);
\node [right] at (2,1/4) {$\phi(Y)$};
\draw [->,color=black!75,very thick,rotate around={20:(0,0)}] (0,0) to ({cos(20)},{sin(20)});

\draw [rotate around={40:(0,0)}, color=gray,very thick] (-1,-1) to[out=60,in=20-180] (0,0) to[out=20,in=180-30] (2,-1/2);
\node [right] at (1.8,1) {$\phi^2(Y)$};
\draw [->,color=gray,very thick,rotate around={40:(0,0)}] (0,0) to ({cos(20)},{sin(20)});

\draw[color=black,very thick] (1.5,-9/8) to[out=145,in=180+60] (0,0) to[out=60,in=180-225] (-1,1);
\node [right] at (1.5,-9/8) {$Z$};

\fill (0,0) circle[radius=2pt];
\node [above left] at (0,0) {$V$};
\end{tikzpicture}
\caption{\(\phi^n(Y)\) is tangent to \(Z\) at \(V\) when \(\phi^n(Y^{(1)}) \subset Z\).}
\end{figure}

More generally, take \(Y_1^\circ = C\) and \(Y_2^\circ = V\), and consider the collection of schemes \(Y^{(j)}\) defined by the ideal sheaves
\[
\mathscr I_{Y_1^\circ} + \mathscr I_{Y_2^\circ}^{j+1},
\]

Then scheme \(Y^{(1)}\) is precisely the one considered above, while \(Y^{(j)}\) corresponds to a \(j\)\textsuperscript{th}-order germ of the curve \(C\) at \(V\).  The collection of closed subschemes \(Y^{(j)}\) is infinite, but it satisfies the hypotheses of Theorem~\ref{nonreducedmordelllangpds}(2).

Applying the theorem to the higher-order germs \(Y^{(j)}\) of \(Y\) at \(V\), 
we find that
\begin{align*}
A_\phi(Y^{(j)},Z) &= \set{n : \phi^n(Y^{(j)}) \subseteq Z} \\
 &= \set{n : \text{$\phi^n(Y)$ is tangent to $Z$ at $V$ to order $\geq j$}} 
\end{align*}
is the union of a finite set and finitely many arithmetic progressions, and there exists a bound \(N\) on the periods of the arithmetic progressions occurring in these sets which is independent of \(j\).  This implies in particular that there is a uniform bound on the order of tangency between \(C\) and \(\phi^n(C)\) for nonzero \(n\): if \(C\) is ever tangent to some \(\phi^n(C)\) to order at least some \(j\), then such a tangency must occur within the first \(N\) iterates or within the finite set.

We now turn our attention to Theorem~\ref{separationthm} in this setting: the  theorem asserts that it that it is possible to blow up above the point \(V\) a finite number of times and reach a model \(X^\prime\) such that \(\phi\) lifts to an automorphism of \(X^\prime\) and the curves \(\phi^n(Y)\) simultaneously become disjoint above \(V\). These two objectives are in tension: if we do not blow up enough, the curves \(\phi^n(C)\) will not be disjoint; on the other hand, if we blow up too aggressively, the automorphism \(\phi\) will not lift. There are essentially two obstacles to constructing the requisite blow-ups.

\medskip

\emph{First issue.} Suppose, for example, that \(C\) in tangent to \(\phi^n(C)\) at \(V\) precisely when \(n\) is a perfect square.  After blowing up the point \(V\), the curves \(C, \phi(C), \phi^4(C),\ldots\) all meet at a single point \(V^\prime\) on the exceptional divisor, while the other \(\phi^n(C)\) do not.  A further blow-up at \(V^\prime\) is necessary in order to make the curves \(\phi^n(C)\) disjoint.  However, the point \(V^\prime\) is not fixed by \(\psi : \Bl_V X \to \Bl_V X\), since the strict transforms of some of the \(\phi^n(C)\) do not pass through it.  As a result, it is impossible to blow up \(V^\prime\) and obtain a model to which $\psi$ lifts.

This issue can be avoided only if \(C\) is tangent to \(\phi^n(C)\) for either all \(n\) or for no \(n\).  More generally, since the statement of the theorem allows us to replace \(\phi\) by a positive iterate, it is necessary that the set of \(n\) for which \(C\) is tangent to \(\phi^n(C)\) is an arithmetic progression. Similar considerations on higher blow-ups suggest that it is also necessary that the set of \(n\) for which \(\phi^n(C)\) is tangent to \(C\) to order \(\geq k\) is an arithmetic progression, for any value of \(k\).

\medskip

\emph{Second issue.} Suppose that \(C\) is tangent to \(\phi^{2^k n}(C)\) to order \(k\) when \(n\) is odd.  This is consistent with the requirement of (1) that the set of iterates tangent to a given order is an arithmetic progression.  However, no finite sequence of blow-ups can separate all the curves, because the order of tangency between \(C\) and \(\phi^{2^k}(C)\) grows without bound.

\medskip

To overcome these two obstacles, we must show that: (1) \(C\) is tangent to \(\phi^n(C)\) to order \(k\) for all \(n\) in some arithmetic progression; (2) the bound on the length of the progression is independent of \(k\).  This is precisely the content of Theorem~\ref{nonreducedmordelllangpds} where applied to the schemes \(Y^{(j)}\) above.

The next two examples give automorphisms to which Theorem~\ref{mainblowup} is applicable, and in which conclusions of the theorem can easily be seen.

\begin{example}
Let \(X = \P^2\), and let \(\phi : \P^2 \to \P^2\) be an automorphism of infinite order.  Take \(D\) be a general line through a fixed point \(x\).  Then \(D\) has infinite order under \(\phi\), but \(V_D = \bigcup_{n \neq 0} D \cap \phi^n(D) = \set{x}\) is Zariski closed.

This is consistent with the last part of Theorem~\ref{mainblowup}: take \(\pi : Y \to X\) to be the blow-up at \(x\).  Then \(\phi\) lifts to an automorphism \(\psi\) of \(Y\), and the divisors \(\psi^n(\tilde{D})\) are pairwise disjoint.  The pencil of lines through \(x\) induces a map \(f : Y \to \P^1\), and there is an automorphism of \(\tau : \P^1 \to \P^1\) with \(f \circ \psi = \tau \circ f\).  Note that \(H^0(D,N_{D/X}) = H^0(D,\cO_D(1)) > 0\) as required.
\end{example}

The next example shows that although \(N_{D/X}\) must be effective, in general it need not have any stronger sort of positivity (e.g.\ nefness or movability).

\begin{example}
Let \(\psi : \P^3 \to \P^3\) be an infinite order automorphism, with fixed points \(p_i\).  Let \(\pi : X \to \P^3\) be the blow-up of \(\P^3\) at the points \(p_1\) and \(p_2\), with exceptional divisors \(E_1\) and \(E_2\). Then \(\psi\) induces an automorphism \(\phi : X \to X\).  Let \(D\) be the strict transform on \(X\) of a general plane in \(\P^3\) passing through the points \(p_1\) and \(p_2\), and let \(L\) be the strict transform of the line between these two points.  Then \(V_D = L\) is not Zariski dense.  

In this case, \(N_{D/X} \cong \left( \pi^\ast \mathscr O_{\P^3}(1) \otimes \cO_X(-E_1-E_2) \right)\vert_D\) is not nef, since it has negative intersection with \(L\).  This shows that the conclusion of  Theorem~\ref{mainblowup} that \(N_{D/X}\) is effective can not be strengthened to include the conclusion that it is nef or even movable.
\end{example}

It follows from Theorem~\ref{mainblowup} that if \(\phi : X \to X\) is an automorphism of a projective variety and \(D \subset X\) is any irreducible divisor which is not the fiber of any \(\phi^k\)-equivariant fibration, then \(\bigcup_{n \neq 0} D \cap \phi^n(D)\) is a Zariski dense subset of \(D\).  However, this set is typically not analytically dense.

\begin{example}
Suppose that \(\phi : X \to X\) is an automorphism of a surface.  Let \(x \in X\) be a attracting fixed point of \(\phi\) with basin of attraction the (analytically) open set \(W^s(x) \subset X\), and let \(y \in X\) be an attracting fixed point of \(\phi^{-1}\) with of attraction \(W^u(y) \subset X\).  Suppose too that \(W^s(x) \cap W^u(y)\) is nonempty. 

Let \(C \subset X\) be a curve meeting \(W^s(x) \cap W^u(y)\) but not passing through the points \(x\) or \(y\).  Choose a point \(z \in C \cap W^s(x) \cap W^u(y)\), and let \(B \subset X\) be a closed ball around \(z\) which is contained in \(W^s(x) \cap W^u(y)\).  
Fix a number \(N\) and compact sets \(K_x\) containing \(x\) and \(K_y\) containing \(y\) with \(K_x \cap C\) and \(K_y \cap C\) empty and such that if \(n > N\) we have \(\phi^n(B) \subset K_x\) and \(\phi^{-n}(B) \subset K_y\).

Let \(B^0\) be the interior of \(B\), and consider the open subset \(A = B^0 \cap C \subset C\).  If \(\abs{n} > N\), and \(a\) is any point of \(A\),  we have either \(\phi^n(a) \in K_x\) or \(\phi^n(a) \in K_y\), so that in particular \(\phi^n(a)\) does not lie on \(C\).  Then \(A \setminus \left( \bigcup_{n=-N}^N \phi^n(C) \cap C\right)\) is an (analytically) open subset of \(C\) which is disjoint from \(\bigcup_{n \neq 0} C \cap \phi^n(C)\), but in general \(C\) will not be a fiber of an invariant rational fibration.
\end{example}

It bears noting that Theorem~\ref{nonreducedmordelllang} contains the classical Skolem--Mahler--Lech theorem on arithmetic progressions as a special case.  As a result, it comes as no surprise that the proof relies on methods of \(p\)-adic analysis.
\begin{example}[\cite{bellghiocatucker}]
Let \(X = \P^2\) with coordinates \([W_0,W_1,W_2]\), and let \(M : \P^2 \to \P^2\) be the linear automorphism given by the matrix
\[
M = \begin{bmatrix}
0 & 1 & 0 \\
0 & 0 & 1 \\
c_3 & c_2 & c_1
\end{bmatrix}
\]
Let \(Y\) be the point \([x_0,x_1,x_2]\), and let \(Z\) be the hyperplane \(W_0 = 0\).  Then \(\phi^n(Y)\) is the point \([x_n,x_{n+1},x_{n+2}]\), where \(x_n\) is defined by the linear recurrence sequence \(x_n = c_1 x_{n-1} + c_2 x_{n-2} + c_3 x_{n-3}\).  Then \(\phi^n(Y) \subseteq Z\) exactly when \(x_n = 0\), and the theorem asserts that the set of such \(n\) is the union of a finite set and an arithmetic progression.  This is the classical Skolem--Mahler--Lech theorem.

Recall that the Skolem--Mahler--Lech theorem is false over a field \(k\) of positive characteristic, and so our main results cannot be extended to that setting.
\end{example}

Automorphisms of blow-ups of \(\P^2\) at configurations of \(9\) or more points are an important source of examples.  However, analogous examples in higher-dimensional settings are little understood.  The result of Theorem~\ref{bcbound} strengthens dimensional constraints on such examples due to Bayraktar and Cantat.

\begin{example}
\label{ellipticfibr}
Let \(C_1\) and \(C_2\) be two general cubics in \(\P^2\), and let \(X\) be the blow-up of \(\P^2\) at the \(9\) points of \(C_1 \cap C_2\), with exceptional divisors \(E_0,\ldots,E_8\).  The pencil of cubics induces an elliptic fibration \(\pi : X \to \P^1\), and the \(E_i\) are sections of the pencil.
The divisor \(E_0\) determines a section of \(\pi\), which we take to be the base point on each fiber.  Fiberwise addition using the group law of the fibers the gives infinite order automorphisms \(\tau_i : X \to X\) (\(1 \leq i \leq 8)\)) induced by addition of \(E_i\).

Taking products \(X = \P^2 \times \P^{n-2}\), we obtain examples of automorphisms of blow-ups of \(X\) along codimension \(2\) subsets which admit automorphisms of infinite order.
\end{example}

The next two examples illustrate some of the subtleties of dynamical Mordell--Lang statements in the setting of non-reduced schemes; the existence of embedded points creates new complications.  The first example gives an automorphism \(\phi : X \to X\) and an infinite collection of closed subschemes \(Y_j \subset X\) such that the periods of the arithmetic progressions in the sets \(A_\phi(Y_j,Z)\) are not uniformly bounded.  This necessitates that we consider only subschemes \(Y_j\) of a restricted form, as in Theorem~\ref{nonreducedmordelllangpds}.

\begin{example}
\label{supportnotenough}
Let \(g : S \to S\) be an automorphism of a projective variety for which there exist periodic points of all periods \(k \geq 0\). For example, we may take \(E\) to be an elliptic curve and let \(S = E \times E\).  Then the map \(g : S \to S\) given by the matrix \(\left(\begin{smallmatrix} 1 & 1 \\ 0 & 1 \end{smallmatrix} \right)\) has periodic points of all possible periods: if \(x_k\) is an \(k\)-torsion point on \(E\), then \(\phi^k(0,x_k) = (kx_k,x_k) = (0,x_k)\), and so this point has period \(k\).

Now, let \(X = S \times S \times \P^1\) and consider the automorphism \(\phi = \id_S \times \phi \times \id_{\P^1} : X \to X\). Let \(Z\) be the subscheme defined by \(\mathscr I_{S \times S \times 0} \cdot \mathscr I_{\Delta \times 0}^2\), giving \(S \times S \times 0\) with nilpotents in the structure sheaf along the diagonal \(\Delta \subset S \times S\).  Then consider the subschemes \(Y_k\) defined by \(\mathscr I_{S \times S \times 0} \cdot \mathscr I_{(y_k,y_k,0)}^2\), where \(y_k\) is a \(g\)-periodic point of period \(k\).  The associated points of \(Y_k\) are the generic point of \(S \times S \times 0\) and  embedded point \((y_k,y_k,0)\).

Observe that \(\mathscr I_{\phi^n(Y_k)} = 
\mathscr I_{S \times S \times 0} \cdot \mathscr I_{\phi(y_k,y_k,0)}^2 = \mathscr I_{S \times S \times 0} \cdot \mathscr I_{(x_k,\phi^n(x_k),0)}^2\), since \(\phi^n(y_k,y_k,0) = (y_k,\phi^n(y_k),0)\).  Then \(\phi^n(Y_k)\) is contained in \(Z\) if and only if the embedded point \((y_k,\phi^n(y_k),0)\) is contained in \(Z\), which is the case exactly when \((y_k,\phi^n(y_k),0)\) lies on the diagonal \(\Delta \times 0 \subset X\), so that \(\phi^n(y_k) = y_k\).  In particular, $\phi^n(Y_k)\subset Z$ only $n$ is a multiple of $k$.
\end{example}

\begin{example}
\label{manyassociatedpoints}
The hypothesis that the set \(\Xi\) is finite in Theorem \ref{nonreducedmordelllangpds} need not hold in general for subschemes defined by ideals of the form \(\sum_{i=1}^r \mathscr I_{Y_i^\circ}^{n_i}\).  This is illustrated by the following example of Singh and Swanson~\cite[Theorem 4.6]{singhswanson}.

Let \(k\) be an algebraically closed field and let \(X\) be the hypersurface in \(\mathbb A^4 = \Spec k[w,x,y,z]\) defined by \(wy^2 + xyz + wz^2 = 0\).  Consider the subschemes defined by the ideal sheaves \(\mathscr I_{Y_1^\circ} = (w)\) and \(\mathscr I_{Y_2^\circ} = (x)\)  It is shown in \cite{singhswanson} that for any \(n\), the associated points of the closed subscheme defined by \(\mathscr I_{Y_1^\circ}^n + \mathscr I_{Y_2^\circ}^n\) 
correspond to the primes
\[
(y,z), (x,y,z), (w,x,y,z), (y,z,x-w \xi - w \xi^{-1}),
\]
where \(\xi\) ranges among the \(n\)\textsuperscript{th} root of unity.  Here we have two surfaces contained in a threefold, intersecting along a curve.  The thickenings of the surfaces defined by powers of the ideals have intersections with various embedded points along the curve.
\end{example}

\section{Dynamical Mordell--Lang in the non-reduced case}
\label{mordelllang}

\begin{definition}
A subset \(A \subset \Z\) is said to be a a one-sided semilinear set if there is a decomposition $$A = F \cup\left(\bigcup_{i=1}^m P_i\right),$$ where \(F\) is a finite set and each \(P_i = a_i + b_i \Z_{\geq 0}\) is a one-sided arithmetic progression.
Similarly, \(A \subset \Z\) is called a (two-sided) semilinear set if $$A =  F \cup\left(\bigcup_i P_i\right),$$ where \(F\) is a finite set and each \(P_i = a_i + b_i \Z\) is a two-sided arithmetic progression.

We say that a one-sided or two-sided semilinear set \(A\) is \(N\)-periodic if there exists a decomposition for which the period \(b_i\) of each \(P_i\) divides \(N\).  \end{definition}
\begin{lemma}
\label{semilinearintersect}
Suppose that \(\set{A_i}\) is a (possibly infinite) collection of \(N\)-periodic semilinear sets.  Then \(\bigcap_i A_i\) is an \(N\)-periodic semilinear set.
\end{lemma}
\begin{proof}
Observe that if \(A\) is an \(N\)-periodic semilinear set, then it can be written as the union of a finite set and a finite set of arithmetic progressions of length exactly \(N\): indeed, if \(k\) divides \(N\), an arithmetic progression of length \(k\) can be subdivided into \(N/k\) progressions of length \(N\).
For each \(A_i\), write
\[
A_i = F_i \cup \left( \bigcup_j P_{i,j} \right)
\]
where each \(P_{i,j}\) is an arithmetic progression with period \(N\).  Then
\[
\bigcap_i A_i = \left( \bigcup_{0 \leq m \leq N} (m+N\Z) \right) \cap \bigcap_i A_i = \bigcup_{0 \leq m \leq N} \left( (m+N\Z) \cap \bigcap_i A_i \right).
\]
Note that \((m+N\Z) \cap A_i = m+N\Z\) if \(m+N\Z\) is among the arithmetic progressions \(P_{i,n}\), while \((m+N\Z) \cap A_i = (m+N\Z) \cap F_i\) is a finite set otherwise.
It follows that \((m+N\Z) \cap A_i = m+N\Z\) is either \(m+N\Z\) or finite, and so
\[
\bigcup_{0 \leq m \leq N} \left( (m+N\Z) \cap \bigcap_i A_i \right)
\]
is a union of finite sets and arithmetic progressions of length \(N\).
\end{proof}

We now prove a local version of our Dynamical Mordell-Lang theorem.

\begin{lemma}[Local Dynamical Mordell--Lang]\label{localmordell-lang}
Let $K$ be a finite extension of $\mathbb{Q}_p$ with valuation ring $R$, uniformizer $\pi$, and residue field $\kappa$, and let $f\in R\langle x_1, \cdots, x_n\rangle^n$ be an $n$-tuple of convergent power series (i.e.\ the $p$-adic absolute values of the coefficients tend to zero) inducing a topological isomorphism $$R[[x_1, \cdots, x_n]]\to R[[x_1, \cdots, x_n]].$$  Suppose that $f$ is affine-linear mod $\pi$.

Then for any two closed formal subschemes $Y, Z$ of $\on{Spf}(R[[x_1, \cdots, x_n]])$, the set of $m\in \mathbb{Z}$ such that $f^m(Y)\subseteq Z$ is a two-sided semi-linear set.  Furthermore, the semilinear set is \(N\)-periodic for some \(N\) depending only on $\#|k|$ and $|\pi|_p$ (and not on $Y$, $Z$, or $f$). 
\end{lemma}
\begin{proof}
By the affine-linearity assumption, $$f(\mathbf{x})=A\mathbf{x}+\mathbf{b}\bmod \pi$$ for some $A\in \GL_n(\kappa)$ and $\mathbf{b}\in \kappa^n.$   As the group of affine-linear transformations of $\kappa^n$ is finite, there exists some $N$ such that $$f^N=\mathbf{x}\bmod \pi.$$ (Observe that this $N$ depends only on $\#|\kappa|$.)  After possibly increasing $N$ (say, replacing it with $p^MN$ for some $M\gg 0$ depending only on $|\pi|_p$), we may assume that $$f^N=\mathbf{x}\bmod \pi^a$$ for any fixed $a$.  See for example Remark 4 of \cite{poonen}.

We may now apply the main result of \cite{poonen} to obtain an element $$g\in R\langle x_1, \cdots, x_n, m\rangle$$ such that $$g(\mathbf{x}, r)=f^{Nr}(\mathbf{x})$$ for any $r\in \mathbb{Z}$.
Given any $\gamma\in R$, $g$ induces a map $$g(-, \gamma): R[[x_1, \cdots, x_n]]\overset{g(\mathbf{x}, \gamma)}{\longrightarrow} R[[x_1, \cdots, x_n]].$$  Fixing any element $h\in \mathscr{I}_Z$, the induced function $$g(h, -): \mathbb{Z}_p\to R[[x_1, \cdots, x_n]]$$ $$g(h, -): \gamma\mapsto g(h, \gamma)$$ is $p$-adic analytic.  Let $q: R[[x_1, \cdots, x_n]]\to \mathscr{O}_Y$ be the natural quotient map. 

Let $m=Nr+s$.  Now $f^{m}(Y) \subseteq Z$ if and only if every function vanishing on \(Z\) vanishes on \(f^m(Y)\), i.e.\ if and only if $q\circ f^s\circ g(h, r)=0$ for all $h\in \mathscr{I}_Z$.  But this last is a $p$-adic analytic function of $r$; hence it either has finitely many  zeros or is identically zero.  It follows that the set of zeros of \(q \circ f^s \circ g(h, -)\) is an \(N\)-periodic semilinear set for every \(h\) in \(\mathscr I_Z\).  By Lemma~\ref{semilinearintersect}, the intersection of these sets for all \(h \in \mathscr I_Z\) is a \(N\)-periodic semilinear set as well. This proves the theorem.
\end{proof}
\begin{remark}
When we say a function $f$ whose target is an $R$-module $M$ of infinite rank (e.g. the function $g(h, -)$ above) is $p$-adic analytic, we mean that for every continuous $R$-linear form $\ell: M\to R$, the composition $\ell\circ f$ is $p$-adic analytic.  
\end{remark}

Before proceeding to the global situations (Theorems \ref{nonreducedmordelllang} and  \ref{nonreducedmordelllangpds}), we need two lemmas which will allow us to reduce to Lemma \ref{localmordell-lang}.
\begin{lemma}[{\cite[Proposition 2.2]{bellghiocatucker}}]\label{bellghiocatuckerlemma}
Let $K$ be a finite extension of $\mathbb{Q}_p$ with valuation ring $R$, uniformizer $\pi$, and residue field $\kappa$.  Let $X$ be a smooth, finite-type $R$-scheme with geometrically connected fibers.  

Let $\phi: X\to X$ be an \'etale $R$-map, and let $x\in X(\kappa)$ be a point with $\phi(x)=x$.  Then there exists an isomorphism $$\widehat{\mathscr{O}_{X, x}}\simeq R[[x_1, \cdots, x_n]]$$ with $n=\dim(X)$ such that
\begin{enumerate}
\item The automorphism of $R[[x_1, \cdots, x_n]]$ induced by $\phi$ is given by some $f\in R\langle x_1, \cdots, x_n\rangle^n$, and
\item $f$ is affine-linear $\bmod\,\pi$.
\end{enumerate}
\end{lemma}
\begin{proof}
The proof is exactly the same as \cite[Proposition 2.2]{bellghiocatucker}, replacing $p$ with $\pi$.
\end{proof}
\begin{lemma}\label{nbdslemma}
Let $R$ be a discrete valuation ring and $X/R$ a separated, finite-type $R$-scheme.  Let $Y, Z$ be closed subschemes of $X$.  Suppose that $Y$ has exactly one associated point (in particular, $Y$ is irreducible).  Let $y\in Y$ be any closed point with ideal sheaf $\mathfrak{m}_y$; let $y_n\subset Y$ be the closed subscheme associated to $\mathfrak{m}_y^n$.  Then $Y\subset Z$ if and only if $y_n\subset Z$ for all $y$ and $n$.
\end{lemma}
\begin{proof}
The hypothesis is that $\mathscr{O}_X\to\widehat{\mathscr{O}_{Y, y}}/\mathfrak{m}_y^n$ has $\mathscr{I}_{Z}$ in its kernel for all $n$, i.e. that the composition $$\mathscr{I}_Z\to \mathscr{O}_X\to \widehat{\mathscr{O}_{Y,y}}$$ is zero.  We wish to show that $\mathscr{I}_Z$ is in the kernel of the quotient map $\mathscr{O}_X\to \mathscr{O}_Y$.  But consider the diagram
$$\xymatrix{
\mathscr{O}_X \ar[rdd]\ar[r] & \mathscr{O}_Y \ar[d]^{\pi_1}\\
& \mathscr{O}_{Y, y}\ar[d]^{\pi_2}\\
& \widehat{\mathscr{O}_{Y, y}}.
}$$
We observe that $\pi_1$ is injective as $Y$ has one associated point, and $\pi_2$ is injective by Krull's intersection theorem.  Thus $$\ker(\mathscr{O}_X\to \widehat{\mathscr{O}_{Y,y}})=\ker(\mathscr{O}_X\to \mathscr{O}_Y),$$ completing the proof.
\end{proof}
\begin{corollary}\label{manyassociatedlemma}
Let $R$ be a discrete valuation ring and $X/R$ a separated, finite-type $R$-scheme.  Let $Y, Z$ be closed subschemes of $X$.   Let ${y_i}\subset Y$ be a collection of closed points with ideal sheaf $\mathfrak{m}_{y_i}$, such that the closure of each associated point of $Y$ contains some $y_i$.  Then $Y\subset Z$ if and only if $Y\cap \text{Spf}(\widehat{\mathscr{O}_{Y, y_i}})\subset Z\cap \text{Spf}(\widehat{\mathscr{O}_{Y, y_i}})$ for all $i$.
\end{corollary}
\begin{proof}
The case where $Y$ has one associated point is exactly Lemma \ref{nbdslemma}.  Without loss of generality, we may suppose that $y_1\subset Y$ has no other associated points in its closure.  Let $\mathscr{I}$ be the ideal sheaf in $\mathscr{O}_Y$ consisting of functions whose support is contained in $\overline{y_1}$.  Then $Y$ is contained in $Z$ if and only if both $\Spec_X(\mathscr{O}_Y/\mathscr{I})$ and $\overline{y_1}$ (the scheme-theoretic image of $y_1$ in $Y$) are contained in $Z$.  But $\overline{y_1}$ has one associated point, and $\Spec_X(\mathscr{O}_Y/\mathscr{I})$ has one fewer associated point than $Y$, so we are done by induction on the number of associated points.
\end{proof}
\begin{proposition}
\label{localassociated}
Let $K$ be a finite extension of $\mathbb{Q}_p$ with valuation ring $R$, uniformizer $\pi$, and residue field $k$.  Let $X$ be a smooth, finite-type $R$-scheme with geometrically connected fibers.  Let $Y,Z$ be closed subschemes of $X$; let $f: X\to X$ be an \'etale $R$-endomorphism.  Then the set of $m \geq 0$ such that $f^m(Y) \subset Z$ is a one-sided semilinear set, whose period depends only on the associated points of $Y$.

Furthermore, if \(f : X \to X\) is an automorphism, then the set of \(m\) such that \(f^m(Y) \subset Z\) is a two-sided semilinear set, with length bounded in the same way. 
\end{proposition}
\begin{proof}
The idea of the proof is to replace $f$ with an iterate so that it fixes a point in the closure of every associated point of $Y$.  Then Lemma \ref{manyassociatedlemma} reduces the claim to a local statement, which we may prove via Lemma \ref{localmordell-lang}.  We now give the details.

Let $r$ be such that for each associated point $y$ of $Y$, $\bar{y}$ contains a $\mathbb{F}_{q^r}$-point.  The map $X(\mathbb{F}_{q^r})\to X(\mathbb{F}_{q^r})$ induced by $f$ has eventual image $$I=\bigcap_n f^n(X(\mathbb{F}_{q^r}))$$ which is permuted by $f$; hence $f^N$ for some $N\gg 0$ fixes the eventual image $I$. Let $R'$ be an unramified extension of $R$ such that each point of $I$ is the image of some $R'$ point of $X$; replace $R$ with $R'$.  Let $g=f^N$; it suffices to show that the set of $m$ such that $g^m(f^i(Y))\subset Z$ for $i=0, \cdots, N-1$ is a semilinear set.

Now given $x\in I$ with ideal sheaf $\mathfrak{m}_x$, fix an isomorphism $$\widehat{\mathscr{O}_{X, x}}=\varprojlim \mathscr{O}_X/\mathfrak{m}_x^n\simeq R[[x_1, \cdots, x_n]].$$  As $f$ is \'etale, one has that $f^N$ induces an isomorphism $$\widehat{\mathscr{O}_{X, x}}\to \widehat{\mathscr{O}_{X, x}}$$ for each $x\in I$.  By Lemma \ref{bellghiocatuckerlemma}, the hypotheses of Lemma \ref{localmordell-lang} are satisfied for this map.  Hence the set of $m$ such that $g^m(f^i(Y))\cap \text{Spf}(\widehat{\mathscr{O}_{X, x}})\subset Z\cap \text{Spf}(\widehat{\mathscr{O}_{X, x}})$ is semilinear for each $x\in I$.  By Lemma \ref{manyassociatedlemma}, this implies the global result we desired.
\end{proof}

\begin{remark}
In fact, the period of the semilinear sets in the statement above does not depend on the supports of the associated points of $Y, Z$, but rather on $X$ and the minimal $k$ such that $$\overline{\xi}(\mathbb{F}_{q^k})\not=\emptyset$$ for all associated points $\xi$ of $Y, Z$.  Here $\overline{\xi}$ denotes the Zariski-closure of $\xi$.  This yields the following corollary, which will prove useful later on.
\end{remark}

\begin{corollary}[Uniform Mahler-Skolem-Lech]\label{uniformmsl}
Let $f: \mathbb{Z}^r\to \mathbb{Z}^r$ be a linear automorphism, and let $x, y\in \mathbb{Z}^r$ be two elements.  Then $$A(x, y):=\{n\mid f^n(x)=y\}$$ is two-sided semilinear with period bounded only in terms of $r$.
\end{corollary}
\begin{proof}
This is immediate from the remark above, after extending scalars to $\mathbb{Z}_3$.  (It can also be checked by easy linear algebra.)
\end{proof}

The results above imply Theorems \ref{nonreducedmordelllang} and \ref{nonreducedmordelllangpds}.
The final step in the argument is to reduce the theorems for varieties over arbitrary fields of characteristic \(0\) to varieties over extensions of the \(p\)-adics.  To do this, it is necessary to produce a single finite-type integral \(\Z\)-algebra \(R\) over which the infinitely many closed subschemes \(Y_j\) are all defined. 

\thmmordelllang*
This is a special case of:
\thmmordelllangpds*

\begin{proof}
The statement is insensitive to base change along a field extension $k\subset k'$, so we may without loss of generality assume $k$ is uncountable.

We first prove claim (1) of Theorem~\ref{nonreducedmordelllangpds}.
Since the schemes \(Y_j\) are all defined by ideals of the form
\[
\mathscr I_{Y_j} = \sum_i \mathscr I_{Y_i^0}^{n_i},
\]
for some finite set of closed subschemes \(Y_i^\circ\), there exists a finite-type integral $\mathbb{Z}$-algebra $R$, $R$-schemes $X', Y_i', Z'$, an \'etale endomorphism $f'$ of $X'$, and a flat map $\iota: \on{Spec}(k)\to \on{Spec}(R)$ such that $X, Y_i, Z, f$ are obtained by base change along $\iota$.  Without loss of generality, $X'$ is smooth with geometrically connected fibers.

The claim of the theorem is then a consequence of Proposition~\ref{localassociated} as follows.  For each $j$, let $Y_j''$ be the subscheme of $X'$ defined by the sheaf of algebras $\mathscr{O}_{Y_j'}/\mathscr{I}$ where $\mathscr{I}$ consists of all functions whose support does not contain the image of $x$ in $X'$.  (This modification serves to eliminate the associated points whose closure does not contain $x$.) Then by hypothesis, the set of associated points of the $\{Y_j''\}$ is containted in the set $\{\xi \in \Xi\mid x\in \overline{\xi}\}$, and hence is finite.  So the $\{Y_j''\}$ have finitely many associated points, and thus by Proposition \ref{localassociated} we conclude that the set of $\{n\mid \phi^n(Y_j'')\subset Z'\}$ is semilinear of period independent of $j$.  Now base-changing to $k$ and restricting to the local ring at $x$ completes the proof.

For claim (2), we proceed as above; as before, there exists a finite-type integral $\mathbb{Z}$-algebra $R$, $R$-schemes $X', Y_i', Z'$, an \'etale endomorphism $f'$ of $X'$, and a flat map $\iota: \on{Spec}(k)\to \on{Spec}(R)$ such that $X, Y_i, Z, f$ are obtained by base change along $\iota$.  Without loss of generality, $X'$ is smooth with geometrically connected fibers.  On the the generic fiber of $X'$, the $\{Y_j'\}$ have only finitely many associated points.  Hence replacing $Y_j'$ with the subscheme of $X'$ defined by $\mathscr{O}_{Y_j'}/\mathscr{I}$, where $\mathscr{I}$ is the ideal sheaf consisting of functions whose support is on the special fiber of $X'$, we may assume that the $\{Y_j'\}$ have finitely many associated points, all of which are on the generic fiber.  Note that $Y_j$ is still the base change of $Y_j'$ along the map $R\to k$.  Now the result follows again from Proposition \ref{localassociated}.
\end{proof}

In general, the set \(\Xi\) of associated points of the closed subschemes defined by \(\sum_{i=1}^r \mathscr{I}_{Y_i^\circ}^{n_i}\) need not be finite, as we saw in Example~\ref{manyassociatedpoints}.  In our applications of Theorem~\ref{nonreducedmordelllangpds}, we will only use part (1) of the theorem, which is generally easier to check.  However, in some circumstances, the stronger  hypothesis needed for part (2) holds automatically.

\begin{theorem}
\label{finiteassociate}
Let \(X\) be a smooth variety, and let \(Y_1^\circ,\ldots,Y_r^\circ\) be a finite set of closed subschemes of \(X\).  Let \(Y_j\) be an infinite set of closed subschemes of the form 
\[
\sum_{i=1}^r \mathscr{I}_{Y_i^\circ}^{n_i},
\]
for integers \(n_i \geq 0\).  If any of the following conditions hold, then the set of associated points of \(Y_j\) is finite.
\begin{enumerate}
\item Each \(Y_i^\circ\) is a Cartier divisor, the \(Y_i^\circ\) intersect transversely, and \(Y_j\) is any collection of subschemes of the form \(\sum_{i=1}^r \mathscr{I}_{Y_i^\circ}^{n_i}\).
\item \(r=2\) and the \(Y_j\) are defined by ideal sheaves
\[
\mathscr I_{Y_1^\circ} + \mathscr I_{Y_2^\circ}^n
\]
\item All but finitely many of the \(Y_j\) are \(0\)-dimensional.
\end{enumerate}
\end{theorem}
\begin{proof}
In each case, it is sufficient to show that there are only finitely many associated points in any affine chart.

Case (1) follows from the observation that if \(R\) is a regular ring and \(f_i\) is a sequence of regular elements, then the set
\[
\bigcup_{n_1,\ldots,n_r} \Ass(R/(f_1^{n_1},\ldots,f_r^{n_r}))
\]
is finite.

For case (2), we must show that if \(R\) is a noetherian ring, and \(I\) and \(J\) are two ideals in \(R\), the set
\[
\bigcup_{n} \Ass(R/(I+J^n))
\]
is finite.  Observe that \(R/(I+J^n) \cong (R/I)/(J^n(R/I))\) as \(R\)-modules.  But for any \(R\)-module \(M\) and ideal \(I\), the associated primes \(\Ass(M/I^nM)\) stabilize for large \(n\)~\cite{brodmann}.

Case (3) is immediate since the only possible associated points are the components of \(Y_j\).
\end{proof}

\section{Local separation by blow-ups}
\label{sectseparation}

We now deduce some geometric consequences of the results of the previous section.
Before beginning the proof of Theorem~\ref{separationthm}, we collect a few preliminary results.  First, it will often be convenient to invoke the following.
\begin{theorem}[{\cite[3.4.1]{kollarresolution}}]
\label{equivariantres}
Suppose that \(\phi : X \to X\) is an automorphism of a singular variety.  There exists a birational map \(\pi : Y \to X\) such that \(Y\) is smooth and \(\phi\) lifts to an automorphism of \(Y\).
\end{theorem}

We now begin the proof of the local result on separation by blow-ups.  
For the remainder of this section, we work in the setting of Theorem~\ref{separationthm}.  Let us fix some notation: we write \(\mathscr I_D\) for the ideal sheaf on \(X\) associated to the divisor \(D\), and \(\mathscr I_V\) for the ideal sheaf of \(V\).  We write \(\mathscr I_{D,V}\) for the ideal in the local ring \(\mathscr O_{X,V}\) determined by \(\mathscr I_D\), and \(\mathfrak m_V\) for the maximal ideal in \(\mathscr O_{X,V}\).   For simplicity, we write \(\phi(\mathscr I)\) for the ideal sheaf \(\im((\phi^{-1})^\ast \mathscr I \to \cO_X)\), so that if \(W\) is a subscheme we have \(\phi(\mathscr I_W) = \mathscr I_{\phi(W)}\).

For each \(k \geq 0\), let \(D^{(k)}\) be the closed subscheme of \(X\) defined by the ideal sheaf \(\mathscr I_D+\mathscr{I}_{V}^{k+1}\).  This is the \(k\)\textsuperscript{th} order germ of \(V\) along \(D\).  Then consider the set
\[
A_k = \set{ n : \phi^n(D^{(k)}) \subseteq D \text{ in $\Spec(\mathscr O_{X,V})$}}.
\]
By the first part of Theorem~\ref{nonreducedmordelllangpds}, applied with \(x\) the generic point of \(V\), there exists an \(N\) so that  each set \(A_k\) is an \(N\)-periodic semilinear set.   The hypothesis that only finitely many associated points of the \(D^{(k)}\) contain \(x\) clearly holds, since the only possible such points are the generic point of \(D\) and \(x\) itself.

To begin, we give some lemmas showing that after replacing \(\phi\) by a suitable iterate the sets \(A_k\) take a particularly simple form. Note that when \(\phi\) is replaced by an iterate \(\phi^n\), the set \(A_k\) is replaced by \(\frac{1}{n}(A_k \cap n \Z)\). 

The next lemma is the key application of Theorem~\ref{nonreducedmordelllangpds}: in the two-dimensional setting, it implies that there is a uniform bound on the order of tangency between \(C\) and the curves \(\phi^n(C)\).  The crucial observation is that because of the uniform bound \(N\) on the periods of the arithmetic progressions in \(A_k\), if any \(\phi^k(D)\) is tangent to \(D\) to order \(j\), this tangency must occur for some \(0 < k \leq N\).

\begin{lemma}
\label{eventuallyfinite}
There exists some \(k\) for which \(A_k\) is a finite set.
\end{lemma}
\begin{proof}
Let \(Y_1^\circ = D\) and \(Y_2^\circ = V\), and consider the collection of subschemes \(D^{(k)}\) defined by the ideal sheaves \(\mathscr I_D + \mathscr I_V^{k+1}\) It follows from Theorem~\ref{nonreducedmordelllangpds}(1) applied to these ideal sheaves, with \(x\) the generic point of \(V\),  that if \(A_k\) is an infinite set, then it contains an infinite arithmetic progression \(P\) with step size dividing \(N\) by Theorem~\ref{nonreducedmordelllangpds}.  In particular, there must exist some \(i\) with \(1 \leq i \leq N\) for which \(\phi^i(D^{(k)}) \subseteq D\) in \(\mathscr O_{X,V}\).

Since \(D\) is not periodic under \(\phi\), the divisors \(D\) and \(\phi^i(D)\) are distinct for any nonzero \(i\), and so
for any \(i\), there is a maximal \(j\) for which \(D^{(j)} \subset D \cap \phi^i(D)\) in \(\mathscr O_{X,V}\).  The claim of the lemma then holds with
\[
k = \max_{1 \leq i \leq N} \left( \max \set{ j : D^{(j)} \subset D \cap \phi^i(D) } \right). \qedhere
\]
\end{proof}

\begin{lemma}
\label{semilineariterate}
Suppose that \(A_0, \ldots, A_n\) is a finite collection of semilinear sets.  Then there exists an integer \(N\) so that \(A_i \cap N \Z\) is either empty, \(\set{0}\), or \(N\Z\) for each \(i\).
\end{lemma}
\begin{proof}
Let \(k\) be an integer divisible by the period of every arithmetic progression appearing in any of the sets \(A_i\).  Each \(A_i\) can then be written as the union of a finite number of residue classes \(k\Z + r\) and a finite set.  Then \(A_i \cap k\Z\) is either all of \(k\Z\) or a finite set for each \(i\).  Replacing \(k\) by a sufficiently large multiple \(N\), we may eliminate any nonzero element of \(A_i \cap k\Z\).
\end{proof}

\begin{lemma}
\label{stepsize}
Suppose that 
\[
\Z = A_0 \supseteq \cdots \supseteq A_{k-1} \supseteq A_k \supseteq A_{k+1} \supseteq \cdots
\]
is a decreasing chain of semilinear sets, so that \(A_k\) is finite for sufficiently large \(k\) and each \(A_k\) contains \(0\).  Then there exist \(n > 0\) and \(k\) such that \(A_i \cap n\Z = n\Z\) for \(i \leq k\) and \(A_i \cap n\Z = \set{0}\) for \(i > k\).
\end{lemma}
\begin{proof}
According to Lemma~\ref{eventuallyfinite}, there exists some \(m\) so that \(A_m\) is a finite set.  By Lemma~\ref{semilineariterate} applied to the sets \(A_0,\ldots,A_m\), there exists \(N\) so that for \(1 \leq i \leq m\), we have either \(A_i \cap N\Z = \set{0}\) or \(A_i \cap N\Z = N\Z\), with \(A_0 \cap N\Z = N\Z\) and \(A_m \cap N\Z = \set{0}\).    Since the sets are decreasing, the claim follows.
\end{proof}

It follows that there exists \(n\) so that after replacing \(\phi\) by the iterate \(\phi^n\), there exists \(k\) with \(A_i = \Z\) for \(i \leq k\), and \(A_i = \set{0}\) for \(i > k\).

\begin{remark}
The geometric significance of Lemma~\ref{stepsize} to the discussion in Section~\ref{examplesection} is this: it might happen, for example, that \(\phi^n(C)\) is tangent to \(C\) at \(x\) whenever \(n\) is is \(0\) or \(1\) modulo \(3\), while \(C\) is tangent to \(\phi^n(C)\) to order \(2\) only for a finite set \(n = 3,6,7\).  In this setting we have \(N = 3\) and take \(k_0 = 1\).  After replacing \(\phi\) by \(\phi^9\), 
we arrange that
\begin{enumerate}
\item \(C\) is tangent to \(\phi^n(C)\) for all \(n\), 
\item \(C\) is not tangent to \(\phi^n(C)\) to order \(2\) for any \(n\).
\end{enumerate}
\end{remark}

 The first step in the proof is to show that \(\mathscr{I}_{D,V} + \mathscr{I}_{\phi^n(D),V} \) stabilizes to an ideal independent of \(n\); the means roughly that the multiplicity of intersection of \(D\) and \(\phi^n(D)\) along \(V\) is independent of \(n\).

We do this first in the case that \(D\) is smooth at the generic point of \(V\).  This follows from calculations in the local ring at \(V\).

\begin{lemma}
\label{localpartregular}
Let \(\phi : X \to X\) be an automorphism of smooth variety over a field of characteristic \(0\), and \(V\) a codimension \(2\) subvariety with \(\phi(V) = V\).  Suppose that \(D\) is an irreducible divisor containing \(V\) and smooth at the generic point of \(V\), and which is not $\phi$-periodic.  Then after replacing \(\phi\) by an iterate, there exists a value of \(k\) so that:
\begin{enumerate}
\item \(\mathscr{I}_{D,V} + \mathscr{I}_{\phi^n(D),V} = \mathscr{I}_{D,V} + \mathfrak m_V^{k+1}\) in \(\cO_{X,V}\) for every nonzero value of \(n\);
\item there exists an ideal sheaf \(\mathscr I_W\) on \(X\) with \(\phi^n(\mathscr I_W) = \mathscr I_W\) for every \(n\), and \(\mathscr I_W\vert_{\mathscr O_{X,V}} = \mathscr{I}_{D,V} + \mathfrak m_V^{k+1}\).
\end{enumerate}
\end{lemma}
\begin{proof}

By Lemma~\ref{stepsize}, we may replace \(\phi\) by an iterate \(\phi^d\) so that there exists a value of \(k\) for which \(A_k = \Z\) while \(A_{k+1} = \set{0}\).  This means that \(\phi^n(D^{(k)}) \subseteq D\) in \(\mathscr O_{X,V}\) for all \(n\), while \(\phi^n(D^{(k+1)}) \subseteq D\) in \(\mathscr O_{X,V}\) only for \(n = 0\).  At the level of ideals in \(\mathscr O_{X,V}\), this yields \(\mathscr{I}_{D,V} \subseteq \mathscr{I}_{\phi^n(D),V} +  \mathfrak m_V^{k+1}\), while \(\mathscr{I}_{D,V} \not\subseteq \mathscr{I}_{\phi^n(D),V} + \mathfrak m_V^{k+2}\). Replacing \(n\) by \(-n\), we have \( \mathscr{I}_{D,V} \subseteq \mathscr{I}_{\phi^{-n}(D),V} + \mathfrak m_V^{k+1}\).  Applying \(\phi^n\) to both sides yields \(\mathscr{I}_{\phi^n(D),V} \subseteq \mathscr{I}_{D,V} + \mathfrak m_V^{k+1}\). Similarly, \(\mathscr{I}_{\phi^n(D),V} \not\subseteq \mathscr{I}_{D,V} + \mathfrak m_V^{k+2}\). It follows that 
\[
\mathscr{I}_{D,V} + \mathfrak m_V^{k+1} = \mathscr{I}_{\phi^n(D),V} + \mathfrak m_V^{k+1}.
\]
Let \(W \subset X\) be the scheme-theoretic image of 
\[
\Spec \mathscr O_{X,V}/(\mathscr{I}_{D,V} + \mathfrak m_V^{k+1}) \to \Spec \mathscr O_{X,V} \to X,
\]
defined by an ideal sheaf \(\mathscr I_W\).  It must be that \(\phi^n(W)\) is the scheme-theoretic closure of 
\begin{align*}
\phi^n(\Spec \mathscr O_{X,V}/(\mathscr{I}_{D,V} + \mathfrak m_V^{k+1})) &= \Spec \mathscr O_{X,V}/(\mathscr{I}_{\phi^n(D),V}+ \phi^n(\mathfrak m_V^{k+1})) \\ &= \Spec \mathscr O_{X,V}/(\mathscr{I}_{\phi^n(D),V} + \mathfrak m_{V}^{k+1}) =  \Spec \mathscr O_{X,V}/(\mathscr{I}_{D}+ \mathfrak m_{V}^{k+1})
\end{align*}
This shows that \(W\) is itself \(\phi\)-invariant.  Note that \(\mathscr{I}_{W,V} = \mathscr{I}_{D,V} + \mathfrak m_V^{k+1}\). 

Since \(X\) is smooth, \(\cO_{X,V}\) is a two-dimensional regular local ring.   The ideal \(\mathscr{I}_{D,V} \subset \cO_{X,V}\) is principal, and since \(D\) is smooth at the generic point of \(V\), \(\mathscr{I}_{D,V}\) is generated by an element not contained in \(\mathfrak m_V^2\).  Then \(\cO_{X,V}/\mathscr{I}_{D,V}\) is a one-dimensional regular local ring, hence a discrete valuation ring.  The ideals in \(\cO_{X,V}\) containing \(\mathscr{I}_{D,V}\) correspond to ideals in \(\cO_{X,V}/\mathscr{I}_{D,V}\), which are of the form \(\mathfrak m_V^r\).  It follows that every ideal in \(\cO_{X,V}\) containing \(\mathscr{I}_{D,V}\) is of the form \(\mathscr{I}_{D,V} + \mathfrak m_V^r\) for some \(r\).

We claim next that
\[
\mathscr{I}_{D,V} + \mathscr{I}_{\phi^n(D),V} = \mathscr{I}_{D,V} + \mathfrak m_V^{k+1}
\]
for every nonzero value of \(n\).  On one hand, we have
\[
\mathscr{I}_{D,V} + \mathscr{I}_{\phi^n(D),V} \subseteq \mathscr{I}_{D,V} + \mathscr{I}_{\phi^n(D),V} + \mathfrak m_V^{k+1} = \mathscr{I}_{D,V} + \mathfrak m_V^{k+1}
\]
To see that \(\mathscr{I}_{D,V} + \mathscr{I}_{\phi^n(D),V} = \mathscr{I}_{D,V} + \mathfrak m_V^{k+1}\), it suffices to note that \(\mathscr{I}_{D,V} + \mathscr{I}_{\phi^n(D),V} \not\subseteq \mathscr{I}_{D,V} + \mathfrak m_V^{k+2}\) by the choice of \(k\).
\end{proof}
We now prove a special case of Theorem \ref{separationthm}; we will later reduce the theorem to this case.

\begin{theorem}[Local separation, smooth case]
\label{separationthmsmooth}
Let \(X\) be a smooth variety over \(k\) and \(\phi : X \to X\) an automorphism.  Suppose that \(D \subset X\) is an irreducible divisor, containing a codimension-\(2\) irreducible closed subset \(V\) with \(\phi(V) = V\), that $D$ is not $\phi$-periodic, and that \(D\) is smooth at the generic point of \(V\).  Then there exists a birational morphism \(\pi  : Y \to X\) such that after replacing \(\phi\) by a suitable iterate:
\begin{enumerate}
\item \(Y\) is smooth;
\item some iterate of \(\phi\) lifts to an automorphism \(\psi : Y \to Y\);
\item \(\pi(\psi^m(\tilde{D}) \cap \psi^n(\tilde{D}))\) does not contain \(V\) for any \(m \neq n\). 
\end{enumerate}
\end{theorem}

\begin{proof}[Proof of Theorem~\ref{separationthmsmooth}]
Let \(\pi_0 : Y_0 \to X\) be the blow-up of \(X\) along the closed subscheme \(W\) obtained from Lemma \ref{localpartregular}.   We claim that
\begin{enumerate}
\item an iterate of \(\phi\) lifts to an automorphism \(\psi_0 : Y_0 \to Y_0\);
\item \(\pi_0(\phi^m(\tilde{D}) \cap \phi^n(\tilde{D}))\) does not contain \(V\) for \(m \neq n\).
\end{enumerate}
The first claim is immediate since the center \(W\) is invariant under some iterate $\phi^r$ of \(\phi\), by construction.  Write \(\psi_0 : Y_0 \to Y_0\) for this lift of $\phi^r$.

By (2) of Lemma~\ref{localpartregular}, we have \(\mathscr I_{D,V} + \mathscr I_{\phi^n(D),V} = \mathscr I_{D,V} + \mathfrak m_V^{k+1}\) in \(\mathscr O_{X,V}\) for every nonzero value of \(n\).  Since this equality holds at the generic point of \(V\), for every \(n\) there exists an open set \(U_n \subset X\) with \(\left( \mathscr I_{D} + \mathscr I_{\phi^n(D)}\right)\vert_{U_n} =  \mathscr I_W \vert_{U_n}\).  It then follows from~\cite[Exercise 7.12]{hartshorne} that the strict transforms of \(D\) and \(\phi^n(D)\) do not meet above \(U_n\), so that \(\pi_0(D \cap \psi_0^n(D))\) is disjoint from \(U_n\).  This shows the necessary disjointness holds for \(m=0\).  In general, we obtain.
\[
\pi_0(\psi_0^m(\tilde{D}) \cap \psi_0^n(\tilde{D})) = \pi_0(\psi_{0}^{m}(\tilde{D}) \cap \psi_0^{n-m}(\tilde{D})) = \phi^{n-m}(\pi_0(\tilde{D} \cap \psi_0^{n-m}(\tilde{D}))).
\]
This set is consequently disjoint from \(\phi^{n-m}(U_n)\), an open subset of \(V\), as required.  Hence (2) and (3) of Theorem~\ref{separationthmsmooth} are satisfied by the model \(\pi_0 : Y_0 \to X\).  However, the variety \(Y_0\) may not itself be smooth.  Applying Theorem~\ref{equivariantres} to \(\psi_0 : Y \to Y\), we obtain a smooth model \(\pi : Y \to X\) on which all three conditions are satisfied, completing the proof.
\end{proof}

\begin{remark}
Some additional care is needed to prove the local separation result (Theorem \ref{separationthm}) in the case that \(D\) is not smooth at the generic point of \(V\), since the proof of Theorem~\ref{separationthmsmooth} required that \(\mathscr I_{D,V}\) is not contained in \(\mathfrak m_V^2 \subset \mathscr O_{X,V}\).  We now show that it is possible to reduce to the case in which \(\tilde{D}\) is smooth at the generic point of \(V\).

To see the difficulty, suppose that \(X\) is a surface and \(V\) is a point, as \S\ref{examplesection}.  It might be that \(D \subset X\) is a curve with a node at \(D\).  It is possible that one of the local branches of \(D\) has tangent direction fixed by \(\phi\), while the other does not; this is illustrated in Figure~\ref{nodalcurve}.

The automorphism lifts to the blow-up of \(X\) at the node of \(D\).  One point of intersection of \(\tilde{D}\) with the exceptional divisor has infinite orbit.  The other intersection \(x\) is a fixed point of the automorphism.  However, the curve is smooth at \(x\), and Theorem~\ref{separationthmsmooth} can be applied to guarantee the existence of a model on which the iterates are separated.
\end{remark}

\begin{figure}[htb]
  \centering
\includegraphics{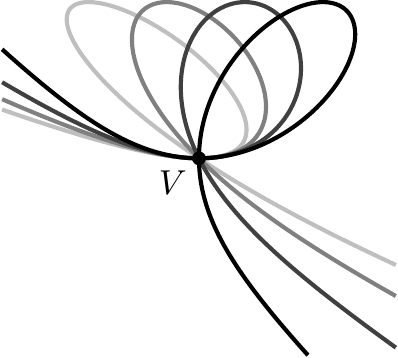}
\caption{Orbit of a curve with a node at the fixed point}
\label{nodalcurve}
\end{figure}

\begin{lemma}
\label{regularizedivisor}
Suppose that \(\phi : X \to X\) is an automorphism of a smooth variety over \(k\) and that \(V \subset X\) is an irreducible codimension \(2\) subvariety with \(\phi(V) = V\).  Suppose that \(D \subset X\) is an irreducible divisor containing \(V\), which is not $\phi$-periodic.  Then there exists a birational map \(\pi : Y \to X\) such that:
\begin{enumerate}
\item \(Y\) is smooth;
\item some iterate of \(\phi\) lifts to an automorphism \(\psi : Y \to Y\);
\item for every value of \(n\), the codimension \(2\) part of \(D \cap \phi^n(D) \cap \pi^{-1}(V) = \bigcup_i V_i\) is a union of finitely many \(\psi\)-invariant codimension \(2\) subvarieties \(V_i\) of \(Y\), and \(D\) is smooth at the generic point of \(V_i\) for each \(i\).
\end{enumerate}
\end{lemma}
\begin{proof}
We will construct \(Y\) via a sequence of blow-ups \(\pi_k^\prime : X_k \to X_{k-1}\) such that \(\phi\) induces automorphisms \(\psi_k : X_k \to X_k\).  Write \(\pi_k : X_k \to X\) for the composition of these blow-ups, and write \(D_k\) for the strict transform of \(D\) on \(X_k\).  

The argument is  by induction on the maximal multiplicity of \(D\)  along a fixed component of \(\psi_k\) lying above \(V\).  More precisely, let \(m(X_k)\) be the number
\[
m(X_k) = \max_{\substack{V_i \subset \pi_k^{-1}(V) \\ \codim_{X_k} V_i = 2 \\ \psi_k^n(V_i) = V_i}} \mult_{V_i}(D_k),
\]
the maximum multiplicity of \(D\) along a \(\psi_k\)-periodic codimension \(2\) subset lying above \(V\).   Here ``\(\psi_k^n(V_i) = V_i\)'' in the subscript means that this equality holds for some non-zero \(n\), so that \(V_i\) is \(\psi_k\)-periodic. Observe that the set of codimension \(2\) subvarieties \(V_i \subset \pi_k^{-1}(V)\) with \(\mult_{V_i}(D_k) \geq 0\) is finite, so this maximum is taken over a finite set.  Note too that replacing \(\phi\) by an iterate does not change \(m(X_k)\), and so we do this freely in what follows.  
Let
\[
n(X_k) = \# \set{V_i \subset \pi^{-1}(V) : \text{$\codim_{X_k} V_i = 2$, $\psi_k^n(V_i) = V_i$, and $\mult_{V_i}(D_k) = m(X_k)$}}
\]
be the number of components achieving this multiplicity.

The proof is by induction on \((n(X_k),m(X_k))\), ordered lexicographically: we show that after a sequence of blow-ups, either the maximal multiplicity decreases, or the number of \(V_i\) achieving this multiplicity decreases.

Suppose that \(m(X_k) \geq 2\), and fix a component \(V_k^0 \subset \pi_k^{-1}(V)\) which is \(\psi_k\)-periodic and has \(\mult_{V_k^0}(D_k) = m(X_k)\).  Let \(\pi_{k+1}^{\prime 0} : X_{k+1}^0 \to X_k\) be the blow-up along \(V_k\).   Since \(V_k\) is \(\psi_k\)-invariant, \(\psi_k\) lifts to an automorphism \(\psi_{k+1}^0 : X_{k+1}^0 \to X_{k+1}^0\).  The variety \(X_{k+1}^0\) might
be singular, but taking a functorial resolution \(\pi_k^\prime : X_{k+1} \to X_k\) we may assume that \(X_{k+1}\) is smooth and that \(\psi_k\) lifts to \(\psi_{k+1} : X_{k+1} \to X_{k+1}\).

Let \(W_1,\ldots,W_r\) be the codimension-\(2\) components of \(D_{k+1} \cap (\pi_{k+1}^\prime)^{-1}(V_k^0)\). Each of the sets \[A_{\psi_{k+1}}(W_i,W_j):=\{z\in \mathbb{Z}\mid \psi_{k+1}^z(W_i)\subset W_j\}\] is semilinear.  By Lemma~\ref{semilineariterate}, after replacing \(\psi_{k+1}\) by an iterate, we may assume that \(A_{\psi_k}(W_i,W_i)\) is either all of \(\Z\) or empty, and that \(A_\phi(W_i,W_j)\) is empty for \(i \neq j\).  

Suppose then that \(D_{k+1} \cap \psi_{k+1}^n(D) \cap \pi_{k+1}^{-1}(V_k^0)\) has a component with codimension \(2\) in \(X\).  It must be that \(\phi^n(W_i) = W_j\) for some \(i\) and \(j\).  But by the above, this is possible only if \(i=j\) and \(W_i\) is \(\phi\)-invariant.

We have \(\mult_{W_i}(D_{k+1}) \leq \mult_{V_k^0}(D_k)\) for any \(i\), so \(m(X_{k+1}) \leq m(X_k)\).  However, it might happen that equality holds for some component \(W_i\), so that \(m(X_{k+1}) = m(X_k)\) and \(n(X_{k+1}) = n(X_k)\).  In this case, we repeat the preceding blow-up procedure, taking \(V_{k+1}^0= W_i\) for the center of the blow-up.  If \(m(X_k)\) and \(n(X_k)\) do not decrease, we continue the process, obtaining a sequence of models
\(\cdots \to X_{k+1} \to X_k\) such that
\begin{enumerate}
\item some iterate of \(\phi\) lifts to an automorphism of \(X_{k+\ell}\) for all \(\ell \geq 0\)
\item \(m(X_{k+1})  = m(X_k)\) and \(n(X_{k+1}) = n(X_k)\).
\end{enumerate}
We claim that such a sequence of blow-ups must be finite: this follows from the standard local calculation showing that embedded resolution of singularities of curves on a surface can be achieved by repeatedly blowing up the singular points.  
Indeed, we may reduce the statement to the two-dimensional case as follows. Fix a surface \(S \subset X\) so that \(S\) is transverse to \(V\).  Write \(S_k \subset X_k\) for the strict transform of \(S\) on \(X_k\).  
Then for every \(\ell \geq k\) we have
\[
\mult_{V_\ell^0 \cap S_\ell}(D_\ell \vert_{S_\ell}) = \mult_{V_\ell^0} D_\ell \geq 2.
\]
Since \(S_{k+1}\) is obtained by blowing up \(S_k\) at the points \(V_k^0 \cap S_k\), this sequence of blow-ups must be finite by~\cite[Lemma 1.68]{kollarresolution}.  This shows that the multiplicity eventually decreases, and we have either \(n(X_\ell) < n(X_k)\) for some \(\ell > k\) (if some other component \(V_k^\prime\) has the same multiplicity as \(V_k\), i.e.\ if \(n(X_k) \geq 2\)) or \(m(X_\ell) < m(X_\ell)\) (if \(n(X_k) = 1\)).  As a result, we eventually reach a model \(Y= X_\ell\) on which \(m(X_\ell) = 1\).

Each model in our sequence has the property that if \(D_k \cap \psi_k^n(D)\) has a component contained in \(\pi_k^{-1}(V)\), this component is \(\phi\)-invariant.  That \(m(X_\ell) = 1\) means that \(D\) is smooth at the generic point of each such component, and so condition (3) holds as required.
\end{proof}

We finally prove Theorem~\ref{separationthm}.  The statement is identical to that of Theorem~\ref{separationthmsmooth}, except it is no longer assumed that \(D\) is smooth at the generic point of \(V\).

\thmlocalsep*

\begin{proof}[Proof of Theorem~\ref{separationthm}]
We first apply Lemma~\ref{regularizedivisor} to obtain a birational map \(\pi_0 : Y_0 \to X\) satisfying conditions (1)--(3) of the lemma, so that \(\phi\) lifts to an automorphism \(\psi_0 : Y_0 \to Y_0\).  Let \(V_1,\ldots,V_r\) be the finitely many codimension-\(2\) \(\psi_0\)-invariant subsets of \(\pi^{-1}(V)\).  According to condition (3), \(\tilde{D}\) is smooth at the generic point of each \(V_i\).  Applying Theorem~\ref{separationthmsmooth} to each \(V_i\), we then obtain to desired blow-up \(\pi : Y \to X\)

To see that condition (4) holds on \(Y\), note that by (3) of Lemma~\ref{regularizedivisor} on \(Y_0\) we have \(\psi_0^m(\tilde{D}) \cap \psi_0^n(\tilde{D}) \cap \pi^{-1}(V)\) equal to the union of the \(V_i\): these intersections have no components other than the \(\psi_k\)-invariant sets.  Then by (3) of Theorem~\ref{separationthmsmooth}, we find that \(\pi(\psi^m(\tilde{D}) \cap \psi^n(\tilde{D}))\) does not contain \(V\).
\end{proof}

\section{Global separation by blow-ups}
\label{sectglobalsep}

The following result on varieties containing infinite sets of disjoint divisors will play a role in the proof of Theorem~\ref{mainblowup}.

\begin{proposition}[{\cite[Theorem 1.1]{disjointdivisors}}]
\label{propdisjoint}
Suppose that \(X\) is a normal projective variety defined over an algebraically closed field, and that \(\set{D_i}\) is an infinite set of pairwise disjoint irreducible divisors on \(X\).  Then there exists a smooth projective curve \(C\) and a map \(f : X \to C\) such that each \(D_i\) is contained in a fiber of \(f\).
\end{proposition}

We also need the following easy lemma.
\begin{lemma}
\label{abhyankar}
Suppose that \(\pi : Y \to X\) is a birational morphism between smooth projective varieties, and that \(f : Y \to C\) is a non-constant morphism to a smooth curve which does not descend to $X$.  Then \(C\) is isomorphic to \(\P^1\).
\end{lemma}
\begin{proof}
Some positive-dimensional fiber of $\pi$ must dominate $C$, as $f$ does not descend to $X$.  But such fibers are rationally connected, so $C$ is $\mathbb{P}^1$.
\end{proof}

We now prove Theorem~\ref{mainblowup}.  The strategy is to repeatedly apply Theorem~\ref{separationthm} to components of \(D \cap \phi^n(D)\), eventually rendering the infinitely many divisors \(\phi^n(D)\) pairwise disjoint.

\begin{theorem}
\label{mainblowupconstruct}
Let \(X\) be a smooth variety over \(k\) and \(\phi : X \to X\) an automorphism.  Suppose that \(D \subset X\) is an irreducible divisor such that  
\[
V(D,\phi) = \set{ d\in D : \text{$\phi^n(d) \in D$ for some nonzero $n$} } = \bigcup_{n \neq 0} D \cap \phi^n(D)
\]
is a proper Zariski-closed subset of $D$.  Then there exists a projective birational morphism \(\pi : Y \to X\) such that
\begin{enumerate}
\item \(Y\) is smooth;
\item some iterate of \(\phi\) lifts to an automorphism \(\psi : Y \to Y\);
\item the divisors \(\psi^n(\tilde{D})\) are pairwise disjoint.
\end{enumerate}
\end{theorem}

\begin{proof}
The proof is by induction on the number of components of \(V(D,\phi)\).  If \(V(D,\phi)\) is empty, claims (1)--(3) of the theorem are already satisfied.  Observe that replacing \(\phi\) by an iterate can serve only to shrink the set \(V(D,\phi)\) and decrease the number of components, and so we freely replace \(\phi\) by iterates in what follows.

We claim now that there exists a sequence \(n_i\) with \(\abs{n_i}\) unbounded, such that \(D \cap \phi^{n_i}(D)\) is nonempty for each \(i\).  Indeed, suppose that there exists \(N\) for which \(\phi^n(D) \cap D\) is empty for all \(n\) with \(\abs{n} > N\).  Replacing \(\phi\) by the iterate \(\phi^N\), we may assume that \(\phi^n(D) \cap D\) is empty for all \(n \neq 0\).  But then \(\phi^m(D) \cap \phi^n(D) = \phi^m(D \cap \phi^{n-m}(D))\) is empty for any \(m \neq n\).  Then claims (1), (2), and (3) hold with \(\pi\) the identity map.

Let \(V_0,\ldots,V_r\) be the finitely many components of \(\bigcup_{n \neq 0} D \cap \phi^n(D)\), so that each \(V_i\) is a codimension-\(1\) subvariety of \(D\).  Since the intersections \(D \cap \phi^{n_i}(D)\) are nonempty, there must exist some irreducible component \(V = V_j\) which is contained in \(D \cap \phi^{n_i}(D)\) for infinitely many of the \(n_i\).

Thus the set \(A_{\phi^{-1}}(V,D) = \set{n : \phi^{-n}(V) \subset D} = \set{n : V \subset \phi^n(D)}\) is infinite. By Lemma~\ref{stepsize}, after replacing \(\phi\) by an iterate, we may assume either that \(\phi^n(V)\) is not contained in \(D\) for any \(n\), or that it is contained in \(V\) for all \(n\).  In the former case, \(V\) is no longer contained in \(V(D,\phi)\), and we continue the induction.  So we may then assume that \(V \subset D \cap \phi^n(D)\) for all \(n\).  Applying \(\phi^{-n}\), we see that \(\phi^{-n}(V) \subset D \cap \phi^{-n}(D)\) for all \(n\).  Since \(\phi^{-n}(V)\) has codimension \(1\) in \(D\), it must coincide with some component \(V_i\), and since there are only finitely many \(V_i\), we must have \(\phi^{-n}(V) = \phi^{-m}(V)\) for some distinct \(m\) and \(n\).  But then \(\phi^{n-m}(V) = V\), and replacing \(\phi\) by the iterate \(\phi^{n-m}\), we may assume that \(\phi(V) = V\).

We now apply the local statement of Theorem~\ref{separationthm} to \(\phi : X \to X\) along the invariant subvariety \(V\) and obtain a model \(\pi : Y \to X\) on which \(\phi\) lifts to an automorphism and for which \(\pi(\psi^m(\tilde{D}) \cap \psi^n(\tilde{D}))\) does not contain \(V\).  Suppose now that \(W \subset V(\tilde{D},\psi)\) is a codimension-\(2\) subvariety of \(Y\).  Then \(W \subset \tilde{D} \cap \psi^n (\tilde{D})\) for some \(n\), and \(\pi(W) \subset D \cap \phi^n(D)\).  This shows that the components of \(V(\tilde{D},\psi)\) are a subset of those of \(V(D,\phi)\).  Since the component \(V\) has been removed and \(\pi\) is an isomorphism away from \(V\), the cardinality of this set decreases.  Applying this procedure inductively, we eventually obtain a model \(\pi : Y \to X\) satisfying conditions (1)--(3).
\end{proof}

\begin{corollary}
\label{disjointdivsapp}
Let \(\phi : X \to X\) be as in Theorem~\ref{mainblowupconstruct}, and let \(\pi : Y \to X\) be the birational model produced by the theorem. Assume in addition that $X$ is projective.

Then there exists a curve \(C\), a morphism \(f : Y \to C\), and an automorphism \(\tau : C \to C\) such that \(f \circ \psi = \tau \circ f\).   Moreover, the normal bundle of \(D\) satisfies \(H^0(D,N_{D/X}) > 0\) and \(D\) moves in a positive-dimensional family on \(X\).  If $V(D,\phi)$ is non-empty, the divisor \(D\) is linearly equivalent to a \(\phi\)-periodic divisor.
\end{corollary}
\begin{proof}
Since the divisors \(\psi^n(\tilde{D})\) are pairwise disjoint, it follows from Proposition~\ref{propdisjoint} that there exists a map \(f : Y \to C\) such that all of the divisors \(\psi^n(\tilde{D})\) are contained in fibers of \(f\).  The map \(f\) has only finitely many reducible fibers, and so there are infinitely many $m$ for which \(\psi^m(\tilde{D})\) and \(\psi^{m+1}(\tilde{D})\) are both irreducible fibers of \(f\). 

At last we show that \(\psi\) induces an automorphism of \(C\).  The map $Y\to C$ is flat by e.g. the miracle flatness theorem, and hence induces a map $C\to \Hilb(Y)$.  Let $B\subset \Hilb(Y)$ be the image of this map; $C$ is the normalization of $B$.  The automorphism \(\psi\) induces an automorphism \(\psi_H : \Hilb(Y) \to \Hilb(Y)\).  Since \(\psi_H([\psi^m(D)]) = [\psi^{m+1}(D)]\), we see that \(\psi_H(B)\) meets \(B\) infinitely many times. Since \(B\) is one-dimensional, it must be that \(B\) is preserved by \(\psi_H\).  Consequently \(\psi\) induces an automorphism of $B$, and hence $C$.  We call this morphism \(\tau : C \to C\); $\tau$ satisfies  \(\phi \circ f = f \circ \tau\) by construction.

The divisor \(\tilde{D}\) is a fiber of \(f :Y \to C\), and so moves in a \(1\)-parameter family inside \(Y\).  It follows that \(D\) moves in a \(1\)-parameter family inside \(X\) and that \(H^0(X,D) > 0\) as claimed.

By Lemma~\ref{abhyankar}, if \(V(D,\phi)\) is non-empty, it must be that \(C\) is isomorphic to \(\P^1\).  It follows that the divisors \(\phi^n(D)\) are all linearly equivalent; as any automorphism of $\mathbb{P}^1$ has a fixed point, this implies that $D$ is linearly equivalent to a periodic divisor.
\end{proof}
Putting together Theorem \ref{mainblowupconstruct} and Corollary \ref{disjointdivsapp}, we have proven Theorem \ref{mainblowup}:
\thmglobalsep*

\section{Applications to automorphism groups}
\label{secblowupautos}

In this section, we will explain how the results of the previous section can be applied in a geometric context to study automorphism groups of varieties constructed by sequences of blow-ups.  Before beginning the proof, we collect some preliminary observations about varieties with multiple contractible divisors.

\begin{lemma}
\label{cantintersect}
Suppose that \(X\) is a smooth projective variety of dimension \(n\), and that \(\pi_i : X \to Y_i\) (\(i \in I\)) are distinct morphisms realizing \(X\) as the blow-up of smooth variety \(Y_i\) at a center of dimension less than or equal to \(r\).  
\begin{enumerate}
\item If \(2r+3 \leq n\), then any two exceptional divisors \(E_i\) and \(E_j\) either coincide or are pairwise disjoint. Moreover, the set of exceptional divisors \(E_i\) is finite.
\item If \(r+3 \leq n\) and \(\Nef(E_0)\) is a rational polyhedral cone, then there exists a finite set of irreducible divisors \(W_i \subset E_0\) such that if \(E_j\) is distinct from \(E_0\), then \(E_0 \cap E_j \subseteq \bigcup_i W_i\) for every \(j \neq 0\).
\end{enumerate}
\end{lemma}
\begin{proof}
First we prove (1).  Suppose that some \(E_i\) and \(E_j\) have nonempty intersection, and choose a point \(x\) be a point in the intersection.  Let \(F \cong \P^{n-r-1}\) be the fiber of \(\pi_i\) containing \(x\) and \(F^\prime \cong \P^{n-r-1}\) be the fiber of \(\pi_j\) containing \(x\).  Since \(F \cap F^\prime\) is nonempty and \(n \geq 2r+3\),  the intersection of \(F\) and \(F^\prime\) must contain a curve \(\Gamma\).  The map \(\pi\) contracts $F$ and hence \(\Gamma\) to a point. But it is impossible to contract a positive-dimensional subvariety of \(F^\prime\) without contracting all of \(F^\prime\), as $F'\simeq \mathbb{P}^{n-r-1}$.  Hence $E_i=E_j$.

Now we prove the finiteness claim.  Let \(f_i\) be a line contained in a fiber of \(E_i\), so that \(E_i \cdot f_i = -1\). Since \(E_i\) is disjoint from \(E_j\) for \(i \neq j\), it must be that \(E_i \cdot f_j = 0\) if \(i \neq j\).  This implies that the classes of the \(E_i\) are linearly independent in \(N^1(X)\).  Since \(N^1(X)\) is finite-dimensional, there can be at most finitely many contractions.

Next we prove (2).  Let \(x\) be any point of \(E_0 \cap E_j\), and let \(F_0\) and \(F_j\) be the fibers of the maps \(\pi_0\) and \(\pi_j\) passing through \(X\).  Observe that \(F_j \cap E_0\) is contracted to a point by  \(\pi_j\), and has dimension \(n-r-2\).  Since \(r+3 \leq n\), this intersection has positive dimension.  It follows that the divisor \(E_0 \cap E_j \subset E_0\) is contracted by the map \(\pi_j\).

Any map \(\pi_i\vert_{E_0} : E_0 \to \pi_i(E_0)\) contracting a divisor determinates a codimension \(1\) face of \(\Nef(E_0)\), and two contractions with different fibers determine different faces of the cone.  
Since \(\Nef(E_0)\) has only finitely many faces by hypothesis, only finitely many divisors \(W_i\) appear among the exceptional divisors of contractions \(\pi_i\vert_{E_0}\).    Hence the support of the intersection \(E \cap \phi^n(E)\) must be contained in the union \(W = \bigcup_i W_i\), for every value of \(n\).
\end{proof}

We will later apply this lemma in the case that \(X\) is a variety with an infinite order automorphism \(\phi : X \to X\), and the maps \(\pi_i : X \to Y_i\) are of the form \(\pi_0 \circ \phi^{-n}\), where \(\pi_0 : X \to Y_0\) is some fixed contraction.  We also note that condition (2) holds automatically if \(\pi : Y \to X\) is the blow-up of \(X\) along any variety of Picard rank \(1\), since then \(E_0\) has Picard rank \(2\).

\begin{lemma}
\label{periodicexceptional}
Suppose that \(\pi : Y \to X\) is the blow-up along a smooth subvariety with exceptional divisor \(E\).  Let \(\phi : Y \to Y\) be an automorphism with \(\phi(E) = E\).  Then the iterate \(\phi^{\rho(E)}\) descends to an automorphism of \(X\).
\end{lemma}
\begin{proof}
This is a case of~\cite[Lemma 3.2]{jdlthreeautos}.  Briefly, it suffices to show that some iterate of \(\phi\) sends fibers of \(E\) to fibers of \(E\), and then \(\phi\) descends to an automorphism of \(X\).  This in turn follows from the fact that \(E\) is a \(\P^n\)-bundle, and a given variety \(E\) has at most \(\rho(E)\) different \(\P^n\)-bundle structures~\cite[Theorem 2.2]{wisniewski}.
\end{proof}

\begin{lemma}
\label{boundperiods}
Suppose that \(\phi : X \to X\) is an automorphism of a smooth projective variety.  There is a constant \(e\) depending only on \(\rho(X)\) (and not on \(\phi\)) such that if \(D\) is a divisor  which is rigid in its cohomology class (in the sense that \(D\) is the only effective divisor with class \([D]\) in \(N^1(X)\)), and \(\phi^n(D) = D\) for some positive \(n\), there exists \(0 < n^\prime \leq e\) with \(\phi^{n^\prime}(D) = D\).
\end{lemma}
\begin{proof}
Since \(\phi^\ast\) and its inverse are both given by invertible integer matrices, their determinant is \(\pm 1\).  Corollary \ref{uniformmsl} provides an integer $e$ depending only on $\rank N^1(X)$ such that \((\phi^\ast)^{n^\prime}([D]) = [D]\) for some \(n^\prime \leq e\).  Since \(D\) is the only divisor with class \([D]\), it must be that \(\phi^n(D) = D\).
\end{proof}

Lemma \ref{boundperiods} does not follow from the bounds on the periods in Theorem~\ref{nonreducedmordelllang}, since the bounds there depend on the field of definition of \(\phi\).  The strategy is instead to conclude the statement from the corresponding statement for the linear map \(\phi^\ast : N^1(X) \to N^1(X)\). 

\thmautomorphisms*
\begin{proof}
We denote the sequence of blow-ups used in constructing \(Y\) as follows.
\[
\xymatrix{
  Y = X_n \ar[r]^-{\pi_{n-1}} &  X_{n-1} \ar[r]^-{\pi_{n-2}} & \cdots \ar[r] & X_1 \ar[r]^-{\pi_0} &  X_0 = X
}
\]

Let \(E_j \subset X_j\) be the exceptional divisor of \(X_j \to X_{j-1}\), and \(\tilde{E}_j \subset X\) its strict transform.  There are two cases: 
\begin{enumerate}
\item all of the divisors \(\tilde{E}_j\) are periodic under \(\phi\), and 
\item some \(\tilde{E}_j\) is not.
\end{enumerate}

As each \(\tilde{E}_j\) is rigid, by Lemma~\ref{boundperiods}, there is a constant \(e(X)\) so that if \(\tilde{E}_j\) is periodic under \(\phi\), it has period dividing \(e(X)\).  We claim that in case (1) the iterate \(\phi^N\) descends to an automorphism of \(X\), where 
\[
N = e(X) \prod_{i=1}^n \rho(E_n)
\]
It follows from Lemma~\ref{boundperiods}, there is a constant \(n_0 = n_0(X)\) such that \(\phi^{r}(E_j) = E_j\) for some \(r\) less than \(n_0\).  Hence replacing \(\phi\) by \(\phi^r\), we may assume that each  divisor \(\tilde{E}_j\) is  invariant under \(\phi\).  Now, since \(E_n \subset X_n\) is invariant under \(\phi : X_n \to X_n\), the iterate \(\phi^{\rho(E_n)}\) descends to an automorphism \(\phi_{n-1} : X_{n-1} \to X_{n-1}\) by Lemma~\ref{periodicexceptional}.  Continuing in this manner, we conclude that \(\phi^N\) descends to an automorphism of \(X\), as claimed.

It remains only to consider case (2), in which some \(\tilde{E}_j\) is not periodic under \(\phi\). We will see that this case is impossible due to hypothesis on the structure of the exceptional divisor.

Let \(j\) be the largest index for which \(E_j\) has infinite orbit.  By the argument above, \(\phi^N\) descends to an automorphism \(\phi : X_j \to X_j\).  Replace \(Y\) by \(X_j\), and let \(\pi : X_j \to X_{j-1}\) be the contraction of \(E_j\), for which we now write \(E\).  Then the maps \(\pi \circ \phi^{-n} : X_j  \to X_{j-1}\) give an infinite set of contractions, all with distinct exceptional divisors \(\phi^n(E)\).  In the case  \(2r + 3 \leq n\), this is an immediate contradiction by Lemma~\ref{cantintersect}.

If instead \(r+3 \leq n\) and \(\Nef(E)\) is polyhedral, applying Lemma~\ref{cantintersect} as before, we conclude that \(\phi^n(E) \cap E\) is a union of the finitely many components \(W_i \subset E\), independent of \(n\).  This means that the set \(V(E) = \bigcup_{n \neq 0} E \cap \phi^n(E)\) is either empty or contains only the components of \(W\).  By Theorem~\ref{mainblowup}, it must be that \(H^0(E,N_{E/X}) > 0\).  This is a contradiction, since \(E\) is the exceptional divisor of a blow-up.
\end{proof}

\corblowupauto*
\begin{proof}
Since all the blow-ups are assumed to be along points and curves, they satisfy condition (2) of Theorem~\ref{generalblowup}: \(n=4\) and \(r \leq 1\).  Moreover, each \(E\) has Picard rank at most \(2\) and the polyhedrality of \(\Nef(E)\) is immediate.
\end{proof}

\corboundedautos*

\begin{proof}

Write $Y=\on{Bl}_Z(X)$ and let $G\subset \underline{\on{Aut}}(X)$ be
the closed subgroup scheme consisting of automorphisms of $X$ that
preserve $Z$.  By assumption, $\underline{\on{Aut}}(X)$ is of finite
type, so the same is true for $G$.  Let $M=\#|G/G^0|$, where $G^0$ is
the connected component of the identity in $G$.

Note that $G$ is naturally a subgroup scheme of
$\underline{\on{Aut}}(Y)$.  By Theorem \ref{generalblowup}, there exists an
integer $N$ such that for any $\phi\in \on{Aut}(Y)$, $\phi^N\in G$;
thus $\phi^{NM}$ is in the identity component of
$\underline{\on{Aut}}(Y)$.  In particular,
$\underline{\on{Aut}}(Y)/\underline{\on{Aut}}^0(Y)$ is a group of
bounded exponent.

Thus by Lemma \ref{burnside}, the natural map $$\on{Aut}(Y)\to
\GL(H^*(Y, \mathbb{C}))$$ has finite image.  Let $\on{Aut}_\omega(Y)$
be the kernel of this map.  By a well-known result of Lieberman--Fujiki
\cite[Proposition 2.2]{lieberman}, $\on{Aut}^0(Y)$ has finite index in
$\on{Aut}_\omega(Y)$, which we have just shown has finite index in
$\on{Aut}(Y)$.  Hence $\on{Aut}(Y)$ has finite component group, as
desired.

\end{proof}

\begin{lemma}\label{burnside}

Let $H\subset \GL_n(k)$ be a group of bounded exponent, with $k$ a
field of characteristic zero.  Then $H$ is finite.

\end{lemma}

\begin{proof}

Let $\overline{H}$ be the Zariski closure of $H$ in $\GL_n$.
$\overline{H}$ is a subgroup scheme of $\GL_n$ which is of bounded
exponent in the sense that if $N$ is the exponent of $H$, we have that
the map

$$
\xymatrix{
\overline{H} \ar[rr]^{x\mapsto x^N} & &\overline{H}
}
$$

is the same as the constant map $x\mapsto 1$.  But now consider the
group $$T:=\ker(\overline{H}(k[\epsilon]/\epsilon^2)\to
\overline{H}(k)).$$  This is a torsion-free group (indeed a $k$-vector
space) of exponent $N$; hence it must equal zero.  Thus $\overline{H}$
is zero-dimensional, hence finite.

\end{proof}

\section{Bounds on intersection multiplicities}
\label{highercodimstuff}

The applications in the preceding sections have focused on intersections \(D \cap \phi^n(D)\) of a divisor with its own iterates under an automorphism of a variety.  In fact, many of the same arguments can be applied to intersections with  smaller dimensional varieties.  In the case that the intersection is \(0\)-dimensional, this is a result of Arnold.

\begin{lemma}
\label{maximalcontainment}
Let \(\phi : X \to X\) be an automorphism of a smooth  variety, and let \(V \subset X\) be a subvariety with \(\phi(V) = V\).  Suppose that \(Y\) and \(Z\) are two irreducible closed subvarieties of \(X\) containing \(V\) and with $\codim_X(Z)=1$, \(\codim_X Y + \codim_X Z = \codim_X V\). Suppose further that $Y$ is regular at the generic point of $V$. 

Then there exists some \(K\) such that \(\mathscr I_V^K \subseteq \mathscr I_Y + \mathscr I_{\phi^n(Z)}\) in $\mathscr{O}_{X, V}$ is satisfied for all nonzero \(n\) such that  \(\codim_{\eta_V} Y \cap \phi^n(Z) = \codim_{\eta_V} V\).  (Here $\eta_V$ is the generic point of $V$ and $\mathscr{O}_{X, V}$ is the local ring of $X$ at the generic point of $V$.)

Furthermore, the set of $n$ such that \(\codim_{\eta_V} Y \cap \phi^n(Z) \neq \codim_{\eta_V} V\) is semilinear.
\end{lemma}
\begin{proof}
Let $V_k$ be the scheme defined by $\mathscr{I}_V^k+\mathscr{I}_Y$.  By Theorem \ref{nonreducedmordelllangpds}, the set $$A_k:=\{n\mid \phi^n(V_k)\subset Z\}$$ is semilinear of period $N$, independent of $k$.  Let $$A_\infty=\bigcap_k A_k.$$  By Lemma \ref{semilinearintersect}, $A_\infty$ is also semilinear of period $N$.

As semilinear sets of period $N$ satisfy the descending chain condition, there exists $K$ such that $A_K=A_\infty$.  By the regularity of $Y$ at the general point of $V$, and the fact that $V$ is codimension one in $Y$ (by the hypothesis on $Z$), $\mathscr{O}_{Y, V}$ is a discrete valuation ring. Hence the localization of $Y\cap \phi^n(Z)$ to the generic point of $V$ is either $\Spec(\mathscr{O}_{Y, V})$ or $\Spec(\mathscr{O}_{Y,V}/\mathfrak{m}_V^r)$ for some $r$.

$A_\infty$ is the set of $n$ such that  the localization of $Y\cap \phi^n(Z)$ to the generic point of $V$ is $\Spec(\mathscr{O}_{Y, V})$, or equivalently $\codim_{\eta_V} Y \cap \phi^n(Z) \neq \codim_{\eta_V} V$, proving the last claim of the lemma. 

Likewise $A_k$ is the set of $n$ such that the localization of $Y\cap \phi^n(Z)$ to the generic point of $V$ is contained in $\Spec(\mathscr{O}_{Y, V}/\mathfrak{m}_V^k)$, which is contained in $\Spec(\mathscr{O}_{X, V}/\mathfrak{m}_V^k)$.  Hence \(\mathscr I_V^K \subseteq \mathscr I_Y + \mathscr I_{\phi^n(Z)}\) in $\mathscr{O}_{X, V}$ for all nonzero \(n\) such that  \(\codim_{\eta_V} Y \cap \phi^n(Z) = \codim_{\eta_V} V\), as desired.
\end{proof}

\begin{theorem}[cf.\ \cite{arnold}]\label{arnoldtheorem}
Let \(\phi : X \to X\) be an automorphism of a smooth  variety, and let \(V \subset X\) be a subvariety with \(\phi(V) = V\).  Suppose that \(Y\) and \(Z\) are two Cohen-Macaulay subvarieties of \(X\) containing \(V\), with $\text{codim}(Z)=1, Y$ regular at the generic point of $V$, and with \(\codim_X Y + \codim_X Z = \codim_X V\).  Then 
\[
\set{\mult_V (Y \cap \phi^n(Z)) }
\]
is uniformly bounded in \(n\), among \(n\) with \(V\) a component of \(Y \cap \phi^n(Z)\) (i.e.\ those \(n\) for which this intersection has the expected dimension along \(V\)).
\end{theorem}
\begin{proof}
For \(n\) for which the intersection is dimensionally correct, the intersection multiplicity along \(V\) is defined by
\[
\mult_V(Y,\phi^n(Z)) = \sum_{i \geq 0} (-1)^i \len_{\cO_{X, V}} \Tor_i^{\cO_{X, V}}(\cO_{Y, V}, \cO_{\phi^n(Z), V})
\]

However, by the Cohen-Macaulay assumption on \(Y\) and \(Z\), all terms except that with \(i=0\) vanish, so that 
\[
\mult_V(Y,\phi^n(Z)) = \len_{\cO_{X,V}} \cO_{X,V} / (\mathscr I_{Y,V} + \mathscr I_{\phi^n(Z),V})
\]
Since by Lemma \ref{maximalcontainment}, there exists $k$ such that
\[
\mathfrak m_V^k \subset \mathscr I_{Y,V} + \mathscr I_{\phi^n(Z),V}
\]
for all nonzero \(n\) such that $Y\cap Z$ has the expected dimension at $V$, we have
\[
\mult_V(Y,\phi^n(Z)) = \len_{\cO_{X,V}} \cO_{X,V} / (\mathscr I_{Y,V} + \mathscr I_{\phi^n(Z),V}) \leq 
\len_{\cO_{X,V}} \cO_{X,V} / \mathfrak m_V^k,
\]
a bound independent of \(n\).  This completes the proof.
\end{proof}

It may happen that there are some values of \(n\) for which \(Y \cap \phi^n(Z)\) has larger than expected dimension, so that \(M_n\) does not compute the multiplicity.  Lemma \ref{maximalcontainment} shows that this set is semilinear.  However, if $Z$ is no longer a divisor, we do not know how to prove an analogue of Theorem \ref{arnoldtheorem} except in the case $\dim(V)=0$ (the case covered by Arnol'd). Furthermore, we do not know how to control the set of $n$ such that \(Y \cap \phi^n(Z)\) has larger than expected dimension.
\begin{question}
Suppose $\codim(Y)>1$.  Is the set of \(n\) for which \(Y \cap \phi^n(Z)\) has larger than expected dimension along \(V\) a semilinear set?
\end{question}

\section{Acknowledgements}

We are grateful to Serge Cantat, Roland Roeder, Paul Reschke, and Kevin Tucker for useful conversations.  The first author was supported by NSF Research Training Group grant \#DMS-1246844.  The second author was supported by an NSF Postdoctoral Fellowship.

\singlespacing
\bibliographystyle{amsplain}
\bibliography{refs}

\end{document}